\documentclass[12pt]{amsart}

\usepackage{tikz}

 \newtheorem{theorem}{Theorem}
\newtheorem{proposition}[theorem]{Proposition}
 \newtheorem{definition}[theorem]{Definition}
 \newtheorem{remark}[theorem]{Remark}
 \newtheorem{lemma}[theorem]{Lemma}
 \newtheorem{corollary}[theorem]{Corollary}
 \newtheorem{example}[theorem]{Example}

\title{Wave fronts and caustics in the tropical plane}
\author{Grigory Mikhalkin, Mikhail Shkolnikov}

\newcommand{\Z}{\mathbb Z}
\newcommand{\R}{\mathbb R}
\newcommand{\N}{\mathbb N}
\newcommand{\Q}{\mathbb Q}
\newcommand{\T}{\mathbb T}
\newcommand{\E}{\mathbb E}
\newcommand{\nr}{N_{\mathbb R}}
\newcommand{\mr}{M_{\mathbb R}}
\newcommand{\dd}{\partial}
\newcommand{\fan}{\operatorname{Fan}}
\newcommand{\et}{{\mathbb E}^2_{\mathbb T}}

\newcommand{\tr}{\operatorname{tr}}

\begin{document}
\begin{abstract}
The paper studies intrinsic geometry in the tropical plane. Tropical structure in the real affine $n$-space is determined by the integer tangent vectors. Tropical isomorphisms are affine transformations preserving the integer lattice of the tangent space, they may be identified with the group $\operatorname{GL_n}(\Z)$ extended by arbitrary real translations.
This geometric structure allows one to define {\em wave front propagation} for boundaries of convex domains. Interestingly enough, an arbitrary compact convex domain in the tropical plane evolves to a finite polygon after an arbitrarily small time. The caustic of a wave front evolution is a tropical analytic curve. The paper studies geometry of the tropical wave fronts and caustics.
In particular, we relate the caustic of a tropical angle to the continued fraction expression of its slope, and treat it as a tropical trigonometry notion.
\end{abstract}
\address{Universit\'e de Gen\`eve,  Math\'ematiques, rue du Conseil-G\'en\'eral 7-9, 1205 Gen\`eve, Suisse}
\email{grigory.mikhalkin@unige.ch} 
\address{Institute of Mathematics and Informatics, Bulgarian Academy of Sciences, Acad. G. Bonchev Str., Bl. 8, 1113 Sofia, Bulgaria}
\email{m.shkolnikov@math.bas.bg} 
\thanks{Research is supported in part by 
the Swiss National Science Foundation grants 200400 and 204125 (G.M.), and by the Simons Foundation International grant no. 992227, IMI-BAS (M.S.).}
\maketitle

\section{Introduction}\label{sec_intro}
Recall that the {\em Riemannian metric} is the geometric structure consisting of distinguishing the unit vectors in each tangent space $T_pX$ of a smooth manifold $X$, $p\in X$. If $X=\R^2$ then $T_pX=\R^2$, and the equation ${x^2+y^2}=1$ defines the sphere of unit vectors in each tangent spaces in a translation-invariant way. Using this data we can introduce the norm of tangent vectors, the distance between points, and the angles between directions from the same point.
The resulting geometric space is the classical Euclidean plane $\E^2$. 

What happens if instead of distinguishing the sphere of unit vectors we distinguish {\em the lattice of integer vectors} in each $T_pX$? This encodes the so-called {\em tropical structure} on $X$ \footnote{We can speak of {\em almost tropical structure} where the lattice varies smoothly when $p$ varies, and of just {\em tropical structure} if it is {\em integrable} in the sense that locally it comes from some $\R^n$.}.
In particular, we can do it for $\R^2$ by distinguishing the tangent vectors with both integer coordinates. The result is called {\em tropical plane} $\et\approx\R^2$. The tropical structure in $\E^2_{\T}$ is invariant with respect to any translation as well as by the action of $\operatorname{GL}_2(\Z)$ in $\R^2$. Translations and $\operatorname{GL}_2(\Z)$ generate the group of {\em tropical isomorphisms} $\operatorname{Iso}(\et)$. Clearly, all tropical isomorphisms conserve the area. 

Consider some geometric objects in $\et$. All lines of rational slope in $\et$ are equivalent. In the same time, there are different equivalent classes of lines of irrational slopes. In fact, it is easy to see that these equivalence classes are detected by the (infinite) tails of the continued fractions of their slopes. 

Suppose that $L\subset\et$ is a line of rational slope, and $p_1,p_2\in L$. We may define the {\em tropical distance} between $p_1$ and $p_2$ as the Euclidean distance divided by the Euclidean length of the primitive integer vector in the direction of $L$. Clearly, this distance is conserved under tropical isomorphisms. 

Suppose that $L_1, L_2\subset \et$ are two parallel lines of rational slope. We may measure the distance between $L_1$ and $L_2$ as follows. Consider the parallelogram spanned by two unit (in the sense of tropical distance) intervals, one on $L_1$ and one on $L_2$. Its area is distance between $L_1$ and $L_2$. Once again, this distance is invariant by the action of $\operatorname{Iso}(\et)$. 
This distance allows us to define the evolution of tropical half-plane of rational slope by \eqref{eq_halfplaneshift} and thus the evolution of convex domains in $\E^2_{\T}$ in Definition \ref{def_continousfront}.

It is instructive to compare the resulting tropical wave front evolution with the classical wave front evolution defined by the Eikonal equation as well as its caustic, see \cite{Arnold-caustics}, \cite{ArnoldPDE}. Given a smooth arc $\alpha$ in $\E^2$ with a choice of a side of $\alpha$ (i.e. its {\em coorientation}), we may define its time $t$ evolution by taking all tangent lines to $\alpha$, moving the line Euclidean distance $t$ in the direction of the coorientation, and defining the result of time $t$ evolution of $\alpha$ as the common tangent curve to all the moved lines. Clearly, in this definition we do not need to consider the tangent lines at all points of $\alpha$, it is sufficient to do it in a dense set of points. Definition \ref{def_continousfront} does this, but moves the lines according to the tropical distance we defined above, thus using tropical instead of Euclidean structure. In this paper, we do it only in the case when $\alpha$ is convex. 
It is interesting that, as shown in Theorem \ref{fan-description}, the arbitrary small time tropical wave front evolution of an admissible domain is a polygonal domain, even if the initial front is smooth. Note that in the Euclidean case we have the opposite phenomenon: the arbitrary small time Euclidean wave front evolution is smooth, even if the initial front is polygonal.

The tropical caustic is traced by the vertices of the tropical wave fronts in all times just like the Euclidean caustic is traced by singular points of the Euclidean wave fronts in all times, but in addition, it contains {\em the final locus of the tropical wave front} that can be a point, an interval, a ray or a line (see Proposition \ref{finallocusclass}) in the case if the domain evolves to an empty set in finite time.
In this case, the final locus is the evolution of the wave front in final time $t_\Phi$ that is the maximal time when the result of evolution is non-empty.

While the wave front is polygonal for any time $0<t<t_\Phi$, we may turn it into a tropical curve if we attach to it the trajectories of the vertices of the wave fronts at all times smaller than $t$. The length of the resulting tropical curve is subject to the tropical version of Noether's formula (Theorem \ref{prop_NoetherCaustic}).
Note that for $t=t_\Phi$ this procedure recovers the caustic, it is a tropical curve by Theorem \ref{caustic-trop-curve}.

Part I of the paper discusses the basic notions of tropical trigonometry. Part II introduces and studies the tropical wave front evolution of the boundary of a closed convex domain in the tropical plane. Part III introduces and studies the corresponding caustic. Finally, Part IV relates the caustic of a tropical angle (which consists of several rays) to the continued fraction of its slope.

Study of tropical wave fronts and caustics was initiated in
\cite{Sh-diss}, where these notions were introduced, particularly, in the case of $\Q$-polygons. 
Theorems  \ref{thm-candom} and \ref{caustic-trop-curve} are restatements of results already obtained in \cite{Sh-diss}, but we include their proofs for the convenience of the reader. The formula \eqref{length-noether} was found in \cite{KaSh3}. In this paper we adapt an intrinsic tropico-geometric viewpoint on the wave front propagation, and pay special attention to the wave fronts and caustics of tropical angles (as a tool of tropical trigonometry).

Noticeably, tropical wave fronts appear as special cases of tropical curves that arize in the study of the Bak-Tang-Wiesenfeld sandpile model in the plane. The sandpile model itself is discrete and, accordingly, the evolution takes place at a regular lattice with a small mesh size. Nevertheless, the scaling limit of the sandpile model is governed by the laws of tropical geometry, see \cite{KaSh} and the references therein. We also refer to the paper \cite{KaSh2} where a certain {\em wave operator} is defined on tropical series.

There is another remarkable discretisation of the tropical wave front evolution. Given a convex lattice polygon in the plane, we may consider the convex hull of the lattice points in its interior, and then iterate the process (the so-called {\em interior hull iteration} or {\em onion skin peeling}), see e.g. \cite{Ca} and reference therein. Proposition \ref{prop_defequivdicrcont} asserts that this process agrees with the tropical wave front evolution in integer time. Accordingly, certain characteristics of convex lattice polygons based on the onion skin peeling evolution, such as {\em level} or {\em lifespan}, agree with the corresponding notions, such as {\em final time} or {\em age} that we introduce in this paper for more general closed convex domains in the plane.

In this paper we work with the tropical plane with a distinguished point (the origin). It can be obtained from an abstract lattice $N$, i.e. an Abelian group isomorphic to $\Z^2$, by taking the tensor product $\nr=N\otimes\R$. To obtain the abstract tropical plane $\et$ we need to forget the origin (i.e. to consider the torsor over $\nr$).
Presence of the origin in $\nr$ allow us to conveniently pass to the dual plane $\mr$ by setting $M=\operatorname{Hom}(N,\Z)$ and $\mr=M\otimes\R$.

\ \\

\noindent{\bf\Large Part I. Tropical trigonometry.}
\section{Tropical angles}
Working in $\nr$ is particularly convenient for {\em tropical angles}, also known as {\em tropical cones}, as we can present them in terms of rays originating at the origin.
The tropical angle is a closed strictly convex cone $\Sigma\subset\nr$ considered up to the automorphisms of the lattice $N$.

A tropical angle $\Sigma$ has two boundary rays. There is a number of positive integer characteristic we may associate to $\Sigma$ in the case when both of these rays are of rational slope, i.e. parallel to primitive integer vectors $v,w\in N$. Namely, the {\em determinant} of $\Sigma$ is the absolute value of the area of $v\wedge w$. The {\em width} of $\Sigma$ is the tropical length of $v-w\in N$ i.e. the largest integer $l\in\Z$ such that $\frac{v-w}l\in N$. The {\em height} of $\Sigma$ is the determinant divided by the width. Note that it is an integer number.

\begin{proposition}\label{prop-determ}
All tropical angles of determinant 1 are isomorphic. Also all tropical angles of determinant 2 are isomorphic. For any integer $d\ge 1$ the angle of height $1$ and width $d$ is unique (up to isomorphism).
However, for any $d\ge 3$ there exist angles of determinant $d$ and height greater than 1. 
In particular, there exist non-isomorphic angles of the same determinant.
\end{proposition}
\begin{proof}
Let $\Sigma\subset\nr$ be a tropical angle of determinant $d$.
If $d=1$ then $\Sigma$ is generated by a basis of $N$. All bases are isomorphic.
If $\Sigma$ is of height 1 then we have $v=e_2$ and $w=e_2+de_1$ for a basis $\{e_1,e_2\}$ of $N$. Thus in this case, $d$ determines the isomorphism class of $\Sigma$.
If $d=2$ then the width of $\Sigma$ cannot be 1 by divisibility argument (both $v$ and $w$ are non-divisible by two), so that its height is 1.
If $d>3$ then we can consider $v=-e_2$ and $w=e_2+de_1$. The width of the corresponding angle is the greatest common divisor of $d$ and $2$, and thus, at most, two. Thus the height of the corresponding angle is at least $d/2$.
\end{proof}

\begin{definition}
We call an angle of determinant 1 {\em the tropical right angle}. More generally, we call an angle of height 1 and width $d$ {\em the $A_{d-1}$-angle} (or, alternatively, {\em the $A_{d-1}$-cone}).
\end{definition}
By Proposition \ref{prop-determ}, all tropical right angles, and all $A_{d-1}$-angles are isomorphic. Their existence for any $d\in\N$ is straightforward.

\begin{definition}
Given a tropical angle $\Sigma\subset\nr$, {\em the dual angle} is 
$$\Sigma^*=\{q\in\mr\ |\ q(p)\ge 0\ \forall p\in\Sigma\}\subset\mr.$$
\end{definition}
Clearly, the determinant of an angle and its dual coincide. 
Note that $\mr$ is itself a tropical plane with the origin isomorphic to $\nr$ (though we do not have any preferred isomorphism).
Thus all tropical trigonometric notions apply equally well in $\mr$.
\begin{definition}
{\em The canonical angle} $\Sigma\subset\nr$ of determinant $d$ is the tropical angle dual to an $A_{d-1}$-angle in $\mr$.
\end{definition}
Note that the width of a canonical angle of determinant $d$ is 1 if $d$ is odd and $2$ if $d$ is even.
Therefore, its height is $1$ if and only if $d\ge 2$. We get the following statement (that has implicitly already appeared in the proof of Proposition \ref{prop-determ}).
\begin{proposition}\label{dual-can}
A canonical angle of determinant $d\le 2$ is isomorphic to an $A_{d-1}$-angle if and only if $d\le 2$.  
\end{proposition}

Any tropical angle $\Sigma\subset\nr$ is the convex hull of its two boundary rays $R_0$ and $R_+$. We may define the {\em complementary tropical angle} $\Sigma'\subset\nr$  as the angle spanned by $R_0$ and $-R_+$.
\begin{proposition}
The complementary angle $\Sigma'\subset\nr$ is isomorphic to the dual tropical angle $\Sigma^*\subset\mr$.
\end{proposition}
In particular, the isomorphism class of the complementary angle does not depend on the chosen order of the boundary rays. 
\begin{proof}
Let us choose a Euclidean metric on $\nr$ by declaring a basis of $N$ to be orthonormal.
This choice gives an identification between $\mr$ and $\nr$ where the dual angle is mapped to the angle formed by the orthogonal rays to $R_0$ and $R_+$ in the direction of the half-planes containing the angle $\Sigma$.
Rotating the result by 90 degrees identifies the dual angle and the complementary angle.
\end{proof}

Suppose now that one boundary ray $R_0\subset\dd\Sigma$ is of rational slope, and the other bounday ray $R_+\subset\dd\Sigma$ is arbitrary.
Then we may choose a tropical isomorphism $\nr\approx\R^2$ so that $R_0$ is mapped to the upper vertical half-axis while $R_+$ is contained in the right half-plane. Then the ordinate $s_+$ of the intersection of $R_+$ with the line $\{x=1\}\subset\R^2$ can be used to reconstruct $\Sigma$. The real number $s_+$ is called the {\em dual slope}, or the {\em tropical cotangent}\footnote{The tropical tangent is inverse of this quantity.} of $\Sigma$.
It is easy to see that the tropical cotangent is well defined up to addition of an integer number (which makes it easier to deal with than the tropical tangent).   

\section{Double tropical angle}
Suppose that $R_-\subset\nr$ is a ray in the left half-plane intersecting the line $\{x=-1\}\subset\R^2$ at a point whose ordinate is $s_-$.
\begin{definition}
{\em The double tropical angle} is $$\tr(R_-,R_0,R_+)=-(s_++s_-)\in\R.$$
\end{definition}

The double angle is well-defined: clearly it does not depend on the identification of $\nr$ with $\R^2$ as long as the ray $R_0$ maps to the upper vertical half-line. Indeed, different identifications vary by the map $(x,y)\mapsto (x,y+nx)$. This change decreases $s_+$ by $n$, but increases $s_-$ by the same amount. Note that $R_-$ and $R_+$ are antiparallel if and only if $\tr(R_-,R_0,R_+)=0$. In case if the convex hull of $R_+$ and $R_-$ is a strictly convex cone $\Sigma$, the ray $R_0$ is contained in $\Sigma$ if and only if $s_++s_->0$. In the same time, unlike its sign, the numerical value of $\tr(R_-,R_0,R_+)$ depends not only on $\Sigma$, but also on the position of the ray $R_0$ there.

\begin{remark}\label{angle-sintersect}
The double tropical angle is useful in toric geometry. Let the toric variety $X_ {\mathcal F}$ be defined by a fan ${\mathcal F}\subset\mr$ (composed of rays of rational slope), and $R_-,R_0,R_+$ are three rays in ${\mathcal F}$ that are consequent with respect to the cyclic order, then $\tr(R_-,R_0,R_+)\in\Q$ \footnote{The toric surface $X_ {\mathcal F}$ might be singular, but all of its singular points are quotient singularities by finite group. Therefore, we have the Poincar\'e duality over $\Q$, and the intersection numbers of divisors are well-defined as rational numbers.}
 is the self-intersection of the divisor corresponding to $R_0$ in $X_ {\mathcal F}$. 
\end{remark}

Suppose now that the angle $\Sigma(R_0,R_+)$ generated by $R_0,R_+$, and the angle $\Sigma(R_-,R_0)$ generated by $R_-,R_0$ are both right angles.
Then there exists a unique automorphism $\Phi:\nr\to\nr$ such that $\Phi(R_-)=R_0$ and $\Phi(R_0)=R_+$.
This automorphism is orientation-preserving, and thus defines a conjugacy class in $\operatorname{SL}_2(\Z)$.

Consider the universal covering 
\begin{equation}\label{eq-coverSL}
\widetilde{\operatorname{SL}}_2(\R)\to\operatorname{SL}_2(\R).
\end{equation}
Its base $\operatorname{SL}_2(\R)$ contains $\operatorname{SL}_2(\Z)$ as a subgroup.
The inverse image of $\operatorname{SL}_2(\Z)$ under the covering \eqref{eq-coverSL} is the group $\widetilde{\operatorname{SL}}_2(\Z)$ which is a central extension of $\operatorname{SL}_2(\Z)$.
Recall that the abelianization \begin{equation}\label{eq-twelve}
\operatorname{SL}_2(\Z)/[\operatorname{SL}_2(\Z),\operatorname{SL}_2(\Z)]\approx\Z_{12}
\end{equation}
is a cyclic group of order 12, while the abelianization  \begin{equation}\label{eq-widetilde}
\widetilde{\operatorname{SL}}_2(\Z)/[\widetilde{\operatorname{SL}}_2(\Z),\widetilde{\operatorname{SL}}_2(\Z)]\approx\Z
\end{equation}
is an infinite cyclic group. 

Since the union $\Sigma(R_0,R_+)\cup \Sigma(R_-,R_0)$ is disjoint from the relative interior of the ray $-R_0$,
there is a preferred way for lifting the automorphism $\Phi\in\operatorname{SL}_2(\Z)$ to an automorphism $\widetilde\Phi\in\widetilde{\operatorname{SL}}_2(\Z)\subset\widetilde{\operatorname{SL}}_2(\R)$. Taking the abelianization \eqref{eq-widetilde} we get an integer number $k(\widetilde\Phi)\in\Z$.
\begin{proposition}
We have 
$$k(\widetilde\Phi)=3+\tr(R_-,R_0,R_+).$$
\end{proposition}
In particular, in this case the double tropical angle is integer.
\begin{proof}
Suppose that $R_+=-R_-$. Then $\Phi$ is the rotation by 90 degrees, and thus $k(\widetilde\Phi)=3$, since we have
$$
\begin{pmatrix}
0 & 1\\
-1 & 0
\end{pmatrix}=
\begin{pmatrix}
1 & 1\\
0 & 1
\end{pmatrix}
\begin{pmatrix}
1 & 0\\
-1 & 1
\end{pmatrix}
\begin{pmatrix}
1 & 1\\
0 & 1
\end{pmatrix},
$$
and all three matrices from the right-hand side represent the same conjugacy class mapped to 1 after the abelianization of $\widetilde{\operatorname{SL}}_2(\Z)$.
In the general case we multiply the right-hand side by the 
matrix 
$$
\begin{pmatrix}
1 & 0
\\
-\tr(R_-,R_0,R_+) & 1
\end{pmatrix}
$$
on the right.
\end{proof}
\begin{corollary}\label{first12}
Suppose that ${\mathcal F}$ is a complete fan in $\nr$ given by the union of rays $R_1,\dots,R_n$ (enumerated clockwise) such that all consecutive angles $\Sigma(R_{j-1},R_j)$, $j=1,\dots,n$, and $\Sigma(R_{n},R_1)$ are tropical right angles.
Then
\begin{equation}\label{complete-angle}
3n+\sum\limits_{j=1}^n \tr(R_{j-1},R_j,R_{j+1})=12.
\end{equation}
\end{corollary}
Here we identify $R_{0}=R_n$ and $R_{n+1}=R_1$ for convenience in notation.
\begin{proof}
Taking the product of the elements $\widetilde\Phi$ for all double tropical angles in ${\mathcal F}$ taken in the cyclic order we get the central element of $\widetilde{\operatorname{SL}}_2(\Z)$ corresponding to a complete turn in $\nr$, and thus to 12 by \eqref{eq-twelve}.
\end{proof}

\section{A Noether-type formula in the tropical plane}
Formula \eqref{complete-angle} can be considered as the tropical trigonometric formula for the complete double angle in ${\mathcal F}$. There are many other guises how the same formula enters combinatorics and toric surfaces as the instance of the Noether formula, cf. e.g. \cite{PRV}.
Below we consider another version of this appearance.

Let $R_-,R_+\subset\nr$ be two non-parallel rays from the origin generated by the primitive vectors $\nu_\pm\in N$, and the tropical angle $\Sigma$ be their convex hull.
\begin{proposition}\label{perp-descr}
There are at most two lines $L$ passing through the origin and tropically perpendicular both to $R_-$ and $R_+$.
If $L$ passes through the interior of $\Sigma$ then $\Sigma$ is a canonical angle. 
If $L$ is disjoint from the interior of $\Sigma$ then $\Sigma$ is an $A_{d-1}$-angle for some $d\in\N$.
\end{proposition}
Conversely, for a canonical angle $\Sigma$ the common perpendicular line $L$ is parallel to $\nu_++\nu_-$ while
 for an $A_{d-1}$-angle $\Sigma$ the common perpendicular line $L$ is parallel to $\nu_+-\nu_-$.
\begin{proof}
A primitive vector $\lambda\in N$ forms a tropical right angle with $\nu_\pm$ if and only if the signed area $\lambda\wedge\nu_\pm$ is 1 or $-1$.
This gives us four non-degenerate linear systems of equations on $\lambda$ which may produce up to two distinct common perpendicular lines (as $\lambda$ and $-\lambda$ generate the same line).

Suppose that $L$ passes through the interior of $\Sigma$. Choose a basis in $N$ so that $L$ is a vertical line.
Then the first coordinate of $\nu_\pm$ must be $\pm1$ and thus $\Sigma$ is canonical.
Taking the complement angle if $L$ does not pass through the interior of $\Sigma$, we get that $\Sigma$ is an $A_{d-1}$-angle .
\end{proof}
With the help of Proposition \ref{dual-can} we get the following corollary.
\begin{corollary}
If we have $d>2$ for the determinant $d$ of the angle $\Sigma$ then the common perpendicular line $L$ through the apex of $\Sigma$ is unique.
If $d=1,2$ then there are two common perpendiculars.
\end{corollary}
\begin{definition}
If the common perpendicular $L$ passes through the interior of $\Sigma$ then we call $L$ {\em the bissectrice} of the tropical angle $\Sigma$.
\end{definition}
By Proposition \ref{perp-descr}, a tropical angle $\Sigma$ admitting the bissectrice must be canonical. Later, we see that the tropical caustic multisects any tropical angle into right angles, cf. Corollaries \ref{cone-subdiv} and \ref{An-single}.

Let $B=\dd P$ be the boundary of a (not necessarily convex) polygon $P\subset\nr$ with rational slope sides and canonical angles. For a non-convex vertex $v\in P$ this means that the tropical cone that locally coincides with the complement of $P$ near $v$ is canonical.
By construction, every edge $E$ of $P$ is a common perpendicular to the bissectrices $L_\pm$ at the adjacent vertices. 
By Proposition \ref{perp-descr}, these bissectrices form an $A_{n-1}$-angle for some $n\in\N$ or are parallel. 
If $L_+$ and $L_-$ are parallel, we set $d_E=0$. Otherwise, we consider the half-plane $H$ whose boundary contains the edge $E$ and whose interior contains a small neighbourhood of $E$ in $P$.
We set $d_E=n$ if $L_+\cap L_-\subset H$, and we set $d_E=-n$ if not. For a vertex $v$ of determinant $n$ we set $d_v=n$ if $v$ is convex, and $d_v=-n$ if $v$ is not convex.

\begin{theorem}\label{12n}
For a (not necessarily convex) polygon $P\subset\nr$ with canonical angles we have
$$\sum\limits_v d_v+\sum\limits_E d_E=12,$$
where the first sum is taken over all vertices $v$ of $P$ and the second sum is taken over all of its edges $E$.
\end{theorem}
\begin{proof}
To a vertex $v$ we assign the primitive vector $\alpha_v\in N$ parallel to the bissectrice of $v$ in the outward direction.
To an edge $E$ we assign the primitive vector $\beta_E\in N$ parallel to this edge in the counterclockwise direction (in $P$).
If $v$ is adjacent to $E$ then $(\alpha_v,\beta_E)$ is a basis of $N$ (since $v$ is canonical).
In the abelianization of $\widetilde{\operatorname{SL}}_2(\Z)$, the passage between $(\alpha_v,\beta_{E_-})$ and $(\alpha_v,\beta_{E_+})$
amounts to $d_v$ if $E_\pm$ are two different edges adjacent to $v$.
Similarly, if $v_\pm$ are two vertices adjacent to the same edge $E$ then the passage between $(\alpha_{v_-},\beta_{E})$ and $(\alpha_{v_+},\beta_{E})$ amounts to $d_E$. As in the proof of Corollary \ref{first12}, after one full turn, these numbers add up to 12. 
\end{proof}

This proof generalises in a straightforward way to the case of an immersed canonical closed broken line $B\subset\nr$.
This means that $B$ is composed from the edges parallel to the vectors in $N$, and all the vertices of $B$ are canonical.
We choose a co-orientation of $B$ and then an orientation of $N$ (used in the $\widetilde{\operatorname{SL}}_2(\Z)/[\widetilde{\operatorname{SL}}_2(\Z),\widetilde{\operatorname{SL}}_2(\Z)]$-calculus) gives us the orientation of $B$.
The co-orientation of $B$ allows one to define the sign of the integer numbers $d_v$ and $d_E$ for the vertices $v$ and edges $E$ of $B$.
In this set-up the previous proof gives us the following statement.
\begin{theorem}
For an immersed canonical broken line curve $B\subset\nr$ we have
$$\sum\limits_v d_v+\sum\limits_E d_E=12\operatorname{rot}(B),$$
where 
the first sum is taken over all vertices $v$ of $P$, the second sum is taken over all of its edges $E$,
and $\operatorname{rot}(B)$ is the rotation number of the oriented immersed broken line $B$.
\end{theorem}

\ \\


\noindent{\Large\bf Part II. Tropical wave fronts.}
\section{Lattice polygons and half-planes}
Let $N\approx \Z^2$ be a rank two lattice and 
$$\nr=N\otimes\R\approx\R^2$$ be the corresponding real plane.
We denote by $M=N^*=\operatorname{Hom}(N,\Z)$ the dual lattice and by 
$\mr=M\otimes\R=\nr^*$ the dual plane.
Recall that a {\em lattice polygonal domain} is a proper closed subset
$\Phi\subset\nr$
that can be represented as the convex hull of a collection of the lattice points (the points from $N$). 
If this collection is finite then we say that $\Phi$ is a (finite) lattice polygon.

\begin{definition}
\label{def_discretefront}
The decreasing sequence 
$$\Phi=\Phi[0]\supset\Phi[1]\supset\Phi[2]\supset\Phi[3]\dots,$$ 
such that $\Phi[t]$ is the convex hull of the set of lattice points in the interior of $\Phi[t-1]$ for every natural $t$,
is called the {\em lattice evolution} of a lattice polygon $\Phi$. 
The corresponding sequence of polygonal boundaries 
$$\dd\Phi=\dd\Phi[0],\dd\Phi[1],\dd\Phi[2],\dd\Phi[3]\dots$$ 
is called the {\em polygonal tropical wave front evolution} $\dd\Phi[t]$ in discrete time 
$t\in\{0,1,2,\dots\}$. 
\end{definition}

Note that the lattice evolution of $\Phi$ is a discrete time evolution, i.e. $\Phi[t]$ is only defined for $t\in\{0,1,2,\dots\}$. 

It turns out that this definition can be generalized to other boundaries of convex domains, and to evolution in continuous times. We do it with the help of {\em half-plane propagations}.

A closed (rational slope) half-plane $H\subset\nr$ is given by 
\begin{equation} H=\{p\in\nr:\lambda(p)\geq c\},\label{eq_halfplane}\end{equation}
$c\in\R$ and $0\neq\lambda\in M$. 
Here we furthermore assume that $\lambda\in M$ is a primitive element of the lattice, so that $\lambda\in M$ and $c\in\R$ are uniquely determined by the subset $H\subset\nr$. We refer to $\lambda\in M$ as the {\em slope} of $H$.
The half-plane \eqref{eq_halfplane} is called a {\em lattice} half-plane if $c\in\Z$. 

We define the (continuous time) propagation of a rational slope half-plane \eqref{eq_halfplane} by
\begin{equation}\label{eq_halfplaneshift}
H(t)=\{p\in\nr:\lambda(p)\geq c+t\}.\end{equation} 
This is a continuous time evolution. 

\begin{lemma}\label{lem_qhalfplaneshift}
If $H$ is a lattice half-plane then $H(t)=H[t]$
for any $t\in\{0,1,2,\dots\}.$
\end{lemma}

\begin{proof}
As the evolution from Definition \ref{def_discretefront}  is invariant under affine automorphisms of $N$,
it suffices to verify the statement for the upper half-plane $\{(x,y)\in\R^2\ |\ y\ge 0\}$ where it is obvious.
\end{proof}

We use \eqref{eq_halfplaneshift} as the basic definition in the wave front propagation in tropical geometry. 
Note that the condition of being primitive for $\lambda$ is the analogue of asking the normal vector to have length one in the classical setup of wave fronts.  

\section{Tropical wave front propagation}
More generally, suppose that $\Phi\subset\nr$ is a convex domain. Recall that its {\em support half-plane} $H$ is a half-plane \eqref{eq_halfplane} such that $H\supset\Phi$ and $\partial H\cap\partial\Phi\neq\emptyset$.
In other words, we have $H\supset\Phi$, but $H(\epsilon)\not\subset\Phi$ for any $\epsilon>0$.
We denote by $\mathcal{H}_\Phi$ the set of all support half-planes of $\Phi$ with rational slopes.

\begin{definition}\label{def_admissible}
A domain $\Phi\subset\nr$ is admissible if it is closed, convex and $\mathcal{H}_\Phi\neq\emptyset.$  
\end{definition}

\begin{remark}\label{rem_alladmissible}
A convex and closed domain is not admissible if and only if it contains a line with irrational slope. 
Namely, it is either equal to the whole plane, a half-plane with irrational slope or a strip with irrational slope. These are the only closed convex domains that do not propagate tropical wave fronts.
\end{remark}

We reiterate that we have considered only the propagation of half-planes with rational slope.
\begin{lemma}\label{lem_restorephi}
If $\Phi\subset\nr$ is an admissible domain then
$\Phi=\bigcap\limits_{H\in\mathcal{H}_\Phi}H.$
\end{lemma}
\begin{proof}
Clearly, we have $\Phi\subset \bigcap\limits_{H\in\mathcal{H}_\Phi}H$. For the opposite inclusion we note that if $p\notin\Phi$ then $p$ and $\Phi$ can be separated by a line with rational slope.
\end{proof}

Now we are ready to define the tropical wave front propagation. 
\begin{definition}\label{def_continousfront}
Suppose $\Phi\subset\nr$ is admissible and $t\geq0$ is a number.
The tropical propagation of $\Phi$ in time $t$ is 
\begin{equation}\label{trop-propagation}
\Phi(t)=\bigcap\limits_{H\in\mathcal{H}_\Phi}H(t).
\end{equation}
The tropical wave front of $\partial\Phi$ in time $t$ is defined as $\partial\Phi(t).$
\end{definition}
By definition, $\Phi(t)\subset\nr$ is also an admissible domain.

\begin{remark}
Note that in the definition of half-plane evolution \ref{eq_halfplaneshift} we may take positive or negative values of $t$, and $H(t_1)\neq H(t_2)$ whenever $t_1\neq t_2$.
The formula \eqref{trop-propagation} also makes sense for any real value of $t$, however, its right-hand side simply coincides with $\Phi$ whenever $t\le 0$. This is the reason why we restrict to $t\ge 0$ in this definition. Later, in Proposition \ref{negative-evolution}, we give a different definition for negative time evolution $\Phi(t)$ (which is however not defined for all negative $t$). 
\end{remark}

Let $\Phi\subset\nr$ be an admissible domain and $t>0$ a number.
\begin{lemma}
The domain $\Phi(t)$ is contained in the interior of $\Phi$. 
\end{lemma}
\begin{proof}
Since $\Phi\subset\nr$ is a convex closed domain, for every $p\in\dd\Phi$ there exists a support half-plane $H_p\subset\nr$ such that $\dd H_p\ni p$. 
Suppose that the slope of $H_p$ is rational. Then $H_p(t)\not\ni p$ and thus $\Phi(t)\not\ni p$.

Suppose that the slope of $H_p$ is irrational. Choose a basis in $M$ and the corresponding coordinates in $\nr$. The slope of the boundary line $\dd H_p$ is represented by an irrational number $\rho\in\R$. By the Lagrange-Dirichlet theorem, we can approximate $\rho$ by rational numbers $\frac{a_n}{b_n}$, where $a_n,b_n\in\Z$, $b_n\to +\infty$, and $$|\rho-\frac {a_n}{b_n}|<\frac 1{b_n^2}.$$
Since $\Phi$ is admissible, we may choose a point $q\in\dd H_p\setminus \Phi$.
Then, for sufficiently large $n$ there are support half-planes $H_n$ of slope $\frac{a_n}{b_n}$ not containing $q$.
By the definition of tropical propagation, the half-plane $H_n(t)$ is the half-plane $H_n$ translated vertically by $\frac t{b_n}$. Since the ``horizontal distance" between $q$ and $p$ does not depend on $n$, we have $H_n(t)\not\ni p$, and thus $\Phi(t)\not\ni p$ for large $n$.
\end{proof}

By a {\em polygonal domain} $\Psi\subset\nr$ of rational slopes we mean 
a convex domain such that for every bounded region $R\subset\nr$ the intersection $\Psi\cap R$ may be presented as the intersection of a finite collection of half-planes with rational slope.
The boundary $\dd\Psi$ is a broken line consisting of {\em sides} of $\Psi$.
\begin{proposition}\label{prop_positiveshiftispolygon}
If $\Phi\subset\nr$ is admissible and $t>0$ then $\Phi(t)$ is a polygonal domain.
\end{proposition}
\begin{proof}
Once again, we fix a basis of $M$, so that we have coordinates, and the induced Euclidean metric in $\nr\subset\R^2$.
In particular we can compute the norm $||\lambda||>0$ for the slope $\lambda\in M$ of a half-plane with rational slope.

In a bounded region, we may find $\epsilon>0$ such that the $\epsilon$-neighbourhood of $\Phi(t)$ is disjoint from $\dd\Phi$.
By the definition of $H(t)$, the distance between the boundary line of $H$ and that of $H(t)$ is equal to $\frac{t}{||\lambda||}$. Note that in the formula \eqref{trop-propagation} we may exclude all half-planes with $\frac{t}{||\lambda||}<\epsilon$ as their boundaries are disjoint from $\Phi(t)$. Since there are only finitely many slopes with $||\lambda||<\frac t\epsilon$, the domain $\Phi(t)$ is the intersection of finitely many half-planes with rational slope in any bounded region.
\end{proof}
\begin{corollary}\label{cor_finitenmbrofsides}
If $\Phi$ is compact, then $\Phi(t)$ is a finite polygon for any $t>0$.
\end{corollary}

\section{Dual fan $\fan(\Phi(t))$, and its evolution}
Once we know that $\Phi(t)$, $t>0$, is a polygonal domain with rational slopes, we may consider
its {\em dual fan} $\fan(\Phi(t))$. By definition, $\fan(\Phi(t))\subset\mr$ is the union of rays generated by the slopes $\lambda\in M$ of the support half-planes $H_\lambda=\{\lambda(p)\ge c_\lambda+t\}$ at the sides of $\Psi$.
Note that $\fan(\Phi(t))$ determines the slopes $\lambda$ of the sides of $\Phi(t)$ as the primitive integer vectors in the directions of the rays. 
If we associate to each ray of $\fan(\Phi(t))$ of slope $\lambda\in M$ a number $c_\lambda\in\R$ then we get the {\em extended fan} of $\Phi(t)$ which completely determines $\Phi(t)$.





\begin{theorem}\label{fan-description}
If $t>0$ and $\Phi(t)$ has non-empty interior then $\Phi(t)$ is a canonical tropical domain.
Furthermore, we have $\fan(\Phi(t'))\subset\fan(\Phi(t))$ if $t'>t$. 
\end{theorem}
\begin{proof}
Suppose that $\lambda_3\in M\cap C$ is a lattice point inside the cone $C$ of $\fan(\Phi(t))$ generated by $\lambda_1,\lambda_2\in M$.
By duality, if $\lambda_1,\lambda_2,\lambda_3\in M$ are collinear then for any $t\in\R$ the lines $\dd H_1(t)$, $\dd H_2(t)$, $\dd H_3(t)$ have a common intersection point, $H_j(t)=\{\lambda_j(p)\ge t\}$
Furthemore, if $\lambda_3$ is contained in the interior of the triangle with vertices $O,\lambda_1,\lambda_2$, then 
\begin{equation}
\label{convex-inclusion}
H_{\lambda_3}(t)\not\supset H_{\lambda_1}(t)\cap H_{\lambda_2}(t)
\end{equation}
if and only if $t> 0$.
Conversely, if $\lambda_3$ is in the exterior of the triangle with vertices $O,\lambda_1,\lambda_2$, then \eqref{convex-inclusion} holds if and only if $t<0$.

Suppose that a cone of $\fan(\Phi(t))$, $t>0$, is not an $A_n$-cone. Then we can find $\lambda_3\in M$ in the interior of the triangle with vertices $O,\lambda_1,\lambda_2$. 
Let $H_{\lambda_k,\Phi}$ be the support half-plane to $\Phi$ with the slope $\lambda_j$, $j=1,2,3$.
Note that 
$$\dd H_{\lambda_3,\Phi}\cap H_{\lambda_1,\Phi}\cap H_{\lambda_2,\Phi}\neq\emptyset$$ by definition of the support hyperplane.
But then, since $\lambda_3$ is in the interior of the triangle with vertices $O,\lambda_1,\lambda_2$, we have 
$$\dd H_{\lambda_3,\Phi}(t)\cap H_{\lambda_1,\Phi}(t)\cap H_{\lambda_2,\Phi}(t)\neq\emptyset$$ for any $t>0$.
We get a contradiction since $C$ has to be subdivided by the ray generated by $\lambda_3$.

By the same argument, a ray in $\fan(\Phi(s+t))\setminus\fan(\Phi(t))$ cannot appear in the case when every cone in $\fan(\Phi(t))$ is an $A_n$-cone.
\end{proof}

\section{Huygens principle}
\begin{theorem}[Huygens's principle]\label{prop_additivefront}
For an admissible domain $\Phi\subset\nr$ and $s,t\geq 0$ we have $$(\Phi(t))(s)=\Phi(t+s).$$
\end{theorem}
\begin{proof}
By definition, we have $(H_\lambda(t))(s)=H_\lambda(s+t)$ for any half-plane $H_\lambda\in{\mathcal H}_\Phi$ with a rational slope $\lambda\in M$. 
By Definition \ref{def_continousfront}, $\Phi(t+s)$ is obtained as the intersection of $H_\lambda(s+t)$ for all support half-planes $H_\lambda$, where 
$\lambda$ is a primitive vector in $M$. 
Similarly, $(\Phi(t))(s)$ is obtained as the intersection of time $s$ propagation of the half-planes from ${\mathcal H}_{\Phi(t)}$ corresponding to the same primitive vectors $\lambda$.

Since $H_\lambda(t)\supset\Phi(t)$ by \eqref{trop-propagation}, $H_\lambda(t+\epsilon_\lambda)\in{\mathcal H}_{\Phi(t)}$ for some $\epsilon_\lambda\ge 0$.
This implies that $(\Phi(t))(s)\subset\Phi(t+s)$.
If $\epsilon_\lambda=0$ then $H_\lambda(t)\in{\mathcal H}_{\Phi(t)}$, and the corresponding half-planes in the definitions of $\Phi(t+s)$  and $(\Phi(t))(s)$ coincide.

If $\epsilon_\lambda>0$ then $\lambda$ does not generate a ray in $\fan(\Phi(t))$ since by the proof of Proposition \ref{prop_positiveshiftispolygon}, only finitely many half-planes $H_\lambda(t)$, $H_\lambda\in{\mathcal H}_\Phi$, have boundaries intersecting a sufficiently small $\delta$-neighbourhood of $\Phi$ in any bounded domain, $\delta>0$.
By Theorem \ref{fan-description}, then $\lambda$ also does not generate a ray in $\fan((\Phi(t))(s))$.
But $(\Phi(t))(s)$ is obtained as the intersection of all its support half-planes corresponding to the rays of $\fan((\Phi(t))(s))$, and thus taking the intersection with $H_\lambda(t+\epsilon_\lambda+s)$ in the definition of $(\Phi(t))(s)$ is redundant, and $(\Phi(t))(s)=\Phi(t+s)$.
%
%
%
\end{proof}

\begin{proposition}\label{prop_defequivdicrcont}
If $\Phi\subset\nr$ is a lattice polygonal domain and $t\in{\mathbb N}$ then $\Phi[t]=\Phi(t)$.
\end{proposition}
\begin{proof}
By induction, using the Huygens principle, it suffices to prove the proposition for $t=1$.
Note that for any lattice half-plane $H$, the complement $H\setminus H(1)$ does not have any lattice points in its interior.
Thus $\Phi[1]\subset\Phi(1)$.

For the opposite inclusion, note that whenever $\Phi[1]\neq\emptyset$, any point $p\in \Phi\setminus\Phi[1]$ is a point of a (closed) lattice triangle $T$ contained in the closure of $\Phi\setminus\Phi[1]$. The triangle $T$ may be chosen so that one of its vertices $v$ belongs to $\dd\Phi[1]$, while the side $E$ opposite to $v$ belongs to $\dd\Phi$.
By subdividing this triangle further, if needed, we may assume that $E$ does not contain lattice points other than its endpoints.
Note that $T\setminus\{v\}$ is contained in $\Phi\setminus\Phi[1]$, and thus cannot have lattice points in its interior.
Thus $T$ is an elementary lattice triangle (of area $\frac12$). Therefore, we have $T\cap H(1)=\{v\}$ for the support  half-plane $H$ whose boundary contains $E$. Thus $p\notin\Phi(1)$.

Finally, suppose that $\Phi[1]$ is empty. Then, similarily, any point $p\in\Phi$ is contained in an elementary lattice triangle $T$ with vertices in $\Phi$, and with a side $E$ contained in $\dd\Phi$. By the same reason, $p\notin\Phi(1)$.
\end{proof}

\begin{definition}\label{def_criticaltimes}
The {\em final time} $t_\Phi\in[0,\infty]$ of an admissible domain $\Phi$ is the minimal value $t$ such that $\Phi(t)$ has an empty interior. The {\em critical} time $t\in(0,t_\Phi)$ is a value $t$ such that for any small $\varepsilon>0$,
the dual fans of $\Phi(t-\varepsilon)$ and $\Phi(t+\varepsilon)$ are different.
\end{definition}

\section{Evolution of corresponding toric surfaces}
Polygonal domains with rational slopes can be interpreted as moment domains of symplectic toric surfaces. 
Namely, a polygonal domain $\Phi(t)$, $t>0$, determines a pair $(S(t),\omega(t)),$ consisting of a toric surface
$S(t)$ with a closed non-degenerate 2-form $\omega(t)$ (called the {\em symplectic form}), cf. \cite{CdS} and \cite{GKZ}.
The toric surface $S(t)$ is a complex surface determined by $\fan(\Phi(t))$, and might be singular.
However, by Theorem \ref{fan-description}, the only possible singular points of $S(t)$ are $A_n$-singularities, $n\in\N$.
The 2-form $\omega$ is defined in the complement of the singular point, making $S(t)$ a singular symplectic space, see \cite{GM}.

In the evolution of the symplectic toric surface $(S(t),\omega(t))$ both $S(t)$ and $\omega(t)$ evolve,
but $S(t)$ only changes at the critical times $t$ (in the sense of Definition \ref{def_criticaltimes}).

\begin{proposition}\label{thm_wavefrontsingularities}
For any admissible $\Phi$ and $t\in(0,t_\Phi)$ the toric surface $S(t)$ may have only $A_n$-type singularities. 
For a critical time $t$ and small $\varepsilon>0$, the transformation from $S(t-\varepsilon)$ to $S(t)$ is a collection of contractions (blowing down) of isolated boundary divisors $D_j\subset S(t-\varepsilon)$.
Each exceptional divisor $D_j$ contains at most one singular point of $S(t-\varepsilon)$.
The point resulting in contraction of $D_j$ is a non-singular point of $S(t)$.
\end{proposition}
The first statement in this proposition is a corollary of Theorem \ref{fan-description}. The rest will be proved in Theorem \ref{thm_nonfinal}.

If $0<t<t_\Phi$ is a non-critical time, and $\varepsilon>0$ is small then we have a natural identification between the complex surfaces $S(t\pm\varepsilon)$ and $S(t)$. Thus we may identify the cohomology groups $H^2(S(t\pm\varepsilon))$ and $H^2(S(t))$.
Recall that the canonical class $K_S\in H^2(S;\Z)$ is well-defined whenever $S$ is a surface with $A_n$-singularities
(or also $D_n$ and $E_6$, $E_7$ and $E_8$ singularities which do not appear for $S(t)$).
It might happen that $\Phi(t)$ is unbounded, or even a polygonal domain with infinite number of sides, so that $S(t)$ is non-compact.
The cohomology classes $K_{S(t)}$ and $[\omega(t)]$ belong to the same vector space $H^2(S(t);\R)$ that might be finite or infinite-dimensional, according to the number of sides of $\Phi(t)$.
\begin{theorem}[\cite{Sh-diss}] \label{thm-candom}
Let $0<t<t_\Phi$ be a non-critical time.
The evolution of the cohomology class $[\omega(t)]\in H^2(S(t);\R)$ of the symplectic form $\omega$ is 
described by the differential equation
\begin{equation}\label{eq_canonicalevolution}{d\over dt}[\omega(t)]=2\pi\ K_{S(t)}.\end{equation} 
\end{theorem}
\begin{proof}
Recall that the second homology group with real coefficients $H_2(S(t);\R)$ of the toric surface $S(t)$ can be expressed in terms of the dual fan $\fan(S(t))$. Namely, 
$H_2(S(t);\R)$ is  the 
subspace of the vector space generated by the elements $e_{\lambda_j}$ corresponding to the
primitive vectors $\lambda_j\in M$ in the directions of the rays of $\fan(S(t))$ as finite linear combinations $\sum\limits_{j} a_je_{\lambda_j}$,
subject to the condition
$$\sum\limits_{j} a_j\lambda_j=0\in M.$$
Furthermore, the $\omega(t)$-area of the class
$\sum\limits_{j} a_je_{\lambda_j}\in H_2(S(t);\R)$ is $$2\pi\sum\limits_{j=1}^k a_jc_j(t),$$ if the polygon $\Phi(t)$ is the intersection of half-planes $\{\lambda_j(p)\ge c_j(t)\}$ in $\nr$.
Thus $$\frac{d}{dt}\omega(t)(\sum\limits_{j} a_je_{\lambda_j})
=-2\pi\sum\limits_{j=1}^ka_j.$$
In the same time, the canonical class $K_{\Phi(t)}$ is represented by the boundary divisors of $\Phi(t)$ taken with the negative sign. Therefore, $$K(\sum\limits_{j} a_je_{\lambda_j})=-\sum\limits_{j=1}^ka_j,$$
which implies \eqref{eq_canonicalevolution}.
\end{proof}

\ \\


\noindent{\bf\Large Part III. Tropical caustics.}
\section{The caustic of an admissible domain}\label{sec_caustic}

A tropical caustic of an admissible domain $\Phi\subset\nr$ is the locus of special points $p$ of its tropical wave fronts $\Phi(t)$ at all times $t\ge 0$. A point $p\in\dd\Phi(t)$, $t> 0$, is said to be special if there 
are two distinct support half-planes $H_\lambda$, $H_\mu$ of $\Phi$ such that $p\in \dd H_\lambda(t)\cap\dd H_\mu(t)$. 
There are two cases. 
If $p$ is special and $p\in\dd\Phi(t)$ then $p$ is necessarily a vertex of the polygonal domain $\Phi(t)$. 
On the other hand, all points of $\dd\Phi(t_\Phi)=\Phi(t_\Phi)$ are special.

\newcommand{\kf}{\mathcal{K}_\Phi}
\begin{definition}\label{def_caustic}
Let $\Phi$ be an admissible domain. Its tropical caustic $\kf\subset\Phi^\circ$ is the locus of all vertices of polygons $\Phi(t)$ for all $t\in(0,t_\Phi)$ together with the degenerate polygon $\Phi_{t_\Phi}$ if $t_\Phi<\infty.$
\end{definition}

We introduce the {\em weights} to the points of $\kf$ as follows.
By Theorem \ref{fan-description}, a vertex of the polygon $\Phi(t)$, $t\in(0,t_\Phi)$, is of $A_n$-type. We prescribe to such a vertex the {weight} equal to $n+1$.
Then we prescribe the weight 2 to all points of $\Phi_{t_\Phi}$.

By a tropical analytic curve $\Gamma$ in an open set $\Omega\subset\nr$ we mean a graph in $\Omega$, i.e. a possibly infinite, but locally finite union of edges $E$. Each edge $E$ is a straight interval with a rational slope, that might include or not its endpoints (or be infinite) but must be relatively closed in $\Omega$. All points inside the same edge must have the same weight. Furthermore, this graph has to be balanced at its every vertex, in the sense that the sum of the outward primitive vectors parallel to all adjacent vertices multiplied by their weight is zero, see \cite{ItMiSh}. 


\begin{theorem}[\cite{Sh-diss}]\label{caustic-trop-curve}
The caustic $\kf$ is a tropical analytic curve in the interior $\Phi^\circ$ of $\Phi$. 
\end{theorem}

A particular way to see that, is to realize $\kf$ as a corner locus of a tropical series  $\mathcal{F}_\Phi\colon\Phi\rightarrow [0,\infty)$ that we now define. Recall that a tropical series is the (possibly infinite) expression
\begin{equation}\label{tropicalseries}
F(x,y)=``\sum\limits_{j,k\in\Z} a_{j,k}x^jy^k"=\sup_{j,k\in\Z} (a_{j,k}+jx+ky),
\end{equation}
where the quotation marks mean that the summation operation refers to taking the supremal value while the multiplication refers to taking the sum.
If locally only finitely many of $(j,k)\in \Z^2$ correspond to maximizing terms in \eqref{tropicalseries} then the {\em corner locus} $\Gamma_F$ of the resulting function is locally finite balanced graph, and thus a tropical curve. The weight of a point inside an edge of $\Gamma_F$ comes from the integer length of the interval spun by the indices $(j,k)$ of the maximizing monomial terms, see \cite{ItMiSh}.  

For a primitive (non-divisible) element $\lambda\in M$ we define 
$$a^\lambda_\Phi=\sup\{ \lambda(p)\ |\ p\in\Phi\}.$$
Note that the half-plane $H=\{\lambda\le a^\lambda_\Phi\}$ is the support half-plane to $\Phi$.
For a point $p\in \Phi$ the time $(a^\lambda_\Phi-\lambda(p))$ evolution of the half-plane $H$ is the half-plane $H_{a^\lambda_\Phi-\lambda(p)}=\{\lambda\le \lambda(p)\}$, and we have $\dd H_{a^\lambda_\Phi-\lambda(p)}\ni p$.
In other words, $a^\lambda_\Phi-\lambda(p)$ can be considered as the {\em tropical distance} to $\dd H$ that can also be interpreted as the time when the evolution of $H$ reaches $p$.

\begin{definition}\label{def_series}
Define $\mathcal{F}_\Phi\colon\Phi\rightarrow[0,\infty)$ as \begin{equation}\label{def_theseries}\mathcal{F}_\Phi(p)=\inf_{\lambda}(a^\lambda_\Phi-\lambda(p))=
-\sup_\lambda(\lambda(p)-a_\Phi^\lambda),\end{equation}
where the infimum (and supremum) is taken over all primitive elements $\lambda\in M$.
\end{definition}

Note that the right-hand side of \eqref{def_theseries} is the tropical series of the form \eqref{tropicalseries}.
The value $\mathcal{F}_\Phi(p)$ can be interpreted as the time when the tropical front of $\dd\Phi$ reaches the point $p$.

\begin{remark}\label{rem_admissibleforseries}
Admissibility of $\Phi$ is equivalent to existence 
of primitive elements $\lambda\in M$ such that $a^\lambda_\Phi<+\infty$.
\end{remark}

The following lemma is straightforward from the definition.
\begin{lemma}\label{lem_levelsets}
$\mathcal{F}_\Phi^{-1}(t)=\partial\Phi(t).$
\end{lemma}

The following proposition implies Theorem \ref{caustic-trop-curve}.
\begin{proposition}\label{prop_cornerlocus} The caustic
$\kf$ is the tropical curve defined by $\mathcal{F}_\Phi$, i.e. the
locus of points $p\in\Phi^\circ$ where $\mathcal{F}_\Phi$ is not smooth.
\end{proposition}
\begin{proof}
Consider the time $t$ wave front $\dd\Phi(t)$. By Proposition \ref{prop_positiveshiftispolygon}, locally near $p\in\dd\Phi_t$,
the wave front $\dd\Phi(t)$ is polygonal, while only finitely many terms in \eqref{def_theseries} reach the value $t$. Such a term is unique if and only if $p$ is not a vertex of $\dd\Phi(t)$. Thus the function $\mathcal{F}_\Phi$ is locally linear at $p$ if $p$ is not a vertex and not smooth otherwise. 
\end{proof}
Note that, as a tropical series, $\mathcal{F}_\Phi$ defines the weight on the edges of its tropical locus $\kf$, and this weight agrees with the weight of points on $\kf$ that we have defined via the type of vertices of $\Phi(t)$.

\section{Tropical caustic of a cone}

If $\Phi$ is a half-plane with rational slope, then $\kf=\emptyset$. If $\Phi$ is a strip with rational slope then $\kf$ is the line going through the middle of the strip taken with weight two.
Let us consider the case when $\Phi=\Sigma\subset\nr$ is the (strictly) convex cone bounded by two non-parallel rays emanating from the same point. 

\begin{figure}
\begin{tikzpicture}
\node at (0,0) {$\bullet$};
\node at (0,1) {$\bullet$};
\node at (0,2) {$\bullet$};
\node at (0,3) {$\bullet$};
\node at (0,4) {$\bullet$};
\node at (0,5) {$\bullet$};
\node at (0,6) {$\bullet$};
\node at (0,7) {$\bullet$};

\node at (1,0) {$\bullet$};
\node at (1,1) {$\bullet$};
\node at (1,2) {$\bullet$};
\node at (1,3) {$\bullet$};
\node at (1,4) {$\bullet$};
\node at (1,5) {$\bullet$};
\node at (1,6) {$\bullet$};
\node at (1,7) {$\bullet$};

\node at (2,0) {$\bullet$};
\node at (2,1) {$\bullet$};
\node at (2,2) {$\bullet$};
\node at (2,3) {$\bullet$};
\node at (2,4) {$\bullet$};
\node at (2,5) {$\bullet$};
\node at (2,6) {$\bullet$};
\node at (2,7) {$\bullet$};

\node at (3,0) {$\bullet$};
\node at (3,1) {$\bullet$};
\node at (3,2) {$\bullet$};
\node at (3,3) {$\bullet$};
\node at (3,4) {$\bullet$};
\node at (3,5) {$\bullet$};
\node at (3,6) {$\bullet$};
\node at (3,7) {$\bullet$};

\draw[very thick](0,0)--(3.15,7.35);
\draw[very thick](0,0)--(3.3,0);

\draw(0,0)--(3.15,6.3);
\draw(0,0)--(3.15,3.15);

\node at (3.4,6.4) {3};
\node at (3.4,3.2) {1};

\begin{scope}[yshift=100,xshift=130]

\draw[very thick](0,0)--(7.7,-3.3);
\draw[very thick](0,0)--(0,3.3);

\draw(7,-3)--(1,0)--(0,1);

\node at (0,-3) {$\bullet$};
\node at (0,-2) {$\bullet$};
\node at (0,-1) {$\bullet$};
\node at (0,0) {$\bullet$};
\node at (0,1) {$\bullet$};
\node at (0,2) {$\bullet$};
\node at (0,3) {$\bullet$};

\node at (1,-3) {$\bullet$};
\node at (1,-2) {$\bullet$};
\node at (1,-1) {$\bullet$};
\node at (1,0) {$\bullet$};
\node at (1,1) {$\bullet$};
\node at (1,2) {$\bullet$};
\node at (1,3) {$\bullet$};

\node at (2,-3) {$\bullet$};
\node at (2,-2) {$\bullet$};
\node at (2,-1) {$\bullet$};
\node at (2,0) {$\bullet$};
\node at (2,1) {$\bullet$};
\node at (2,2) {$\bullet$};
\node at (2,3) {$\bullet$};

\node at (3,-3) {$\bullet$};
\node at (3,-2) {$\bullet$};
\node at (3,-1) {$\bullet$};
\node at (3,0) {$\bullet$};
\node at (3,1) {$\bullet$};
\node at (3,2) {$\bullet$};
\node at (3,3) {$\bullet$};

\node at (4,-3) {$\bullet$};
\node at (4,-2) {$\bullet$};
\node at (4,-1) {$\bullet$};
\node at (4,0) {$\bullet$};
\node at (4,1) {$\bullet$};
\node at (4,2) {$\bullet$};
\node at (4,3) {$\bullet$};

\node at (5,-3) {$\bullet$};
\node at (5,-2) {$\bullet$};
\node at (5,-1) {$\bullet$};
\node at (5,0) {$\bullet$};
\node at (5,1) {$\bullet$};
\node at (5,2) {$\bullet$};
\node at (5,3) {$\bullet$};

\node at (6,-3) {$\bullet$};
\node at (6,-2) {$\bullet$};
\node at (6,-1) {$\bullet$};
\node at (6,0) {$\bullet$};
\node at (6,1) {$\bullet$};
\node at (6,2) {$\bullet$};
\node at (6,3) {$\bullet$};

\node at (7,-3) {$\bullet$};
\node at (7,-2) {$\bullet$};
\node at (7,-1) {$\bullet$};
\node at (7,0) {$\bullet$};
\node at (7,1) {$\bullet$};
\node at (7,2) {$\bullet$};
\node at (7,3) {$\bullet$};
\end{scope}

\end{tikzpicture}
\caption{A cone $\Sigma\subset\nr$ (on the left) with its caustic, and its dual cone $\Sigma^*\subset\mr$ with the intervals of $\mathcal{P}_\Sigma$ (on the right).\label{1-3-caustic}}
\end{figure}
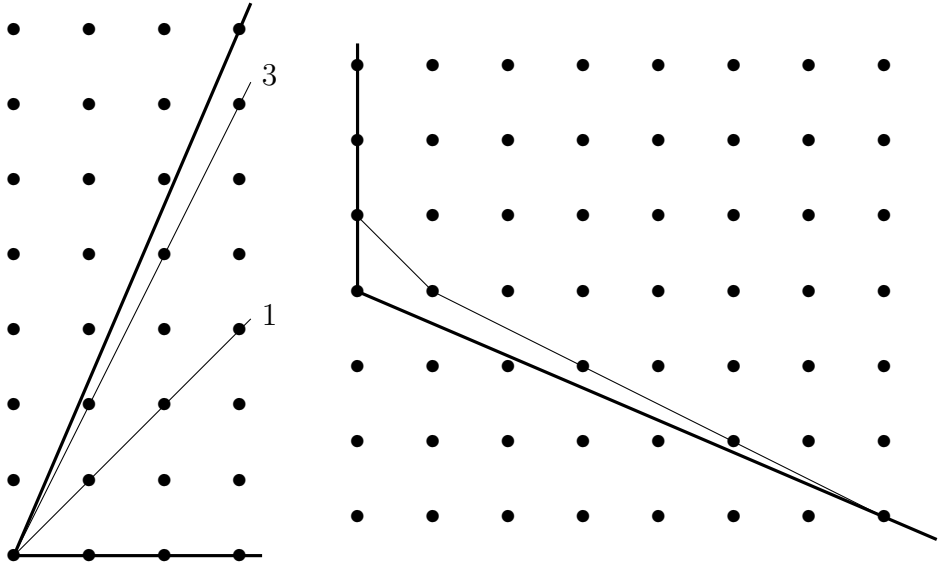

Without loss of generality, we may assume that the apex of the cone is the origin in $\mr$.
and that all the coefficients $a_\lambda^\Sigma$ of the series $\mathcal{F}_\Sigma$ from \eqref{def_theseries} vanish.
In this case the series $\mathcal{F}_\Sigma$ is homogeneous, i.e $\mathcal{F}_\Sigma(cp)=c\mathcal{F}_\Sigma(p),$ for all $p\in\Sigma$ and $c>0$, and the tropical caustic of $\Sigma$ is the (possibly infinite) union of open rays with rational slopes emanating from the origin. 

There is the following recipe for determining these rays, as well as their weights. Take $\Sigma^*\subset\mr$ the dual cone\footnote{As usual, $\Sigma^*=\{\nu\in M_\mathbb{R}:\nu(\Sigma)\ge 0\}$.} to $\Sigma.$ Denote by $\mathcal{P}_\Sigma$ a lattice polygonal domain in $M_\mathbb{R}$ defined as the convex hull of the set $\Sigma^*\cap M \backslash\{0\}.$

\newcommand{\ks}{{\mathcal K}_\Sigma}
\begin{proposition}\label{prop_causticofcone} 
There is a one-to-one correspondence between rays in $\ks$ and finite length sides of $\mathcal{P}_\Sigma:$ for every ray there is a unique side orthogonal to it, and vice versa. Moreover, the lattice length of a side in $\mathcal{P}_\Sigma$ is the weight of a corresponding ray in $\kf$.
\end{proposition}

\begin{proof} 
Consider a lattice point $\nu\in\mathcal{P}_\Sigma\cap M$. 
If $\nu$ is not a vertex of $\mathcal{P}_\Sigma$ then there exist adjacent vertices $\lambda_1,\lambda_2$ of $\mathcal{P}_\Sigma$ and $c_1,c_2\geq 0,$ $c_1+c_2\geq 1$ such that $\nu=c_1\lambda_1+c_2\lambda_2.$ 
For a point $p\in\Sigma\subset\nr$ we have 
$$\nu(p)\stackrel{(A)}{\geq} (c_1+c_2)\min(\lambda_1(p),\lambda_2(p))\stackrel{(B)}{\geq}\min(\lambda_1(p),\lambda_2(p))\stackrel{(C)}{\geq}\mathcal{F}_\Sigma(p).$$  

Note that the inequality (C) holds since $\mathcal{F}_\Sigma$ is defined as the infimal tropical distance to the boundaries of the support half-planes to $\Sigma$, see Definition \ref{def_series}.

If $\nu$ belongs to the interior of $\mathcal{P}_\Sigma,$ then $c_1+c_2>1$ and the inequality (B) is strict for all $p\in\Sigma^\circ$ since $\min(\lambda_1(p),\lambda_2(p))>0.$ In other words, monomials from the interior of $\mathcal{P}_\Sigma$ do not contribute to the series $\mathcal{F}_\Sigma$ on $\Sigma^\circ.$

If $\nu$ belongs to the boundary of $\mathcal{P}_\Sigma,$ then $c_1+c_2=1.$ In such case, $\nu(p)$ can be equal to $\mathcal{F}_\Sigma(p)$ only if (A) is an equality, that is $\lambda_1(p)=\lambda_2(p).$ The locus $\Lambda$ of the points $p\in\Sigma^\circ$ with this property is a ray orthogonal to $\lambda_1-\lambda_2,$ a side $L$ of the polygon $\mathcal{P}_\Sigma.$ 

Note that the support half-planes of $\Sigma$ are defined by the indivisible non-zero vectors of $M\cap\Sigma^*$ (also contained in $\mathcal{P}_\Sigma$), while the vectors from $\mathcal{P}_\Sigma\setminus L$ correspond to support half-planes with boundaries at a distance larger than $\lambda_1(p)=\lambda_2(p).$ By definition, the weight of $\Lambda$ is the lattice length of $L$.
\end{proof}

The following corollaries are immediate.

\begin{corollary}\label{cone-subdiv}
If both sides of $\Sigma$ are parallel to integer vectors then the caustic $\ks$ consists of a finite number of rays. Otherwise, it consists of infinite number of rays. Each cone in the subdivision of $\Sigma$ by these rays is generated by an integer basis of $N$ (i.e. is a tropical right angle).
\end{corollary}

\begin{corollary}\label{An-single}
The caustic $\ks$ of a cone $\Sigma\subset\nr$ consists of a single ray if and only if $\Sigma$ is canonical. The weight of $\ks$ is $n+1$ if $\Sigma$ is dual to an $A_n$-cone.
\end{corollary}

\begin{remark}
The only ray of the caustic of $\Sigma$ in the case when $\Sigma$ is a canonical tropical angle can be interpreted as its {\em tropical bissectrice}, see \cite{MiL}. For a general tropical angle, $\ks$ is its {\em tropical multisectrice} (consisting of several rays).
\end{remark}

Let $\Phi(t)$, $t>0$, be a polygonal domain obtained as the result of time $t$ evolution of an admissible domain $\Phi\subset\nr$, and $E(t)\subset\dd\Phi(t)$ be its side. If $E(t)$ is bounded then it has two endpoints at two canonical angles of $\Phi(t)$. Let $B_+,B_-$ be the corresponding tropical bissectrices of that angles, 
$\nu_+,\nu_-\in M$ be the primitive vectors conormal to $B_\pm$, and oriented in the outward direction with respect to $E(t)$.
Let $\nu\in M$ be the primitive vectors conormal to $E(t)$, and oriented in the outward direction with respect to $\Phi(t)$.
Denote by $R_\pm,R\subset\mr$ the rays generated by $\nu_\pm,\nu$.

The double tropical angle $\tr(R_-,R,R_+)\in\Z$ is integer since ${\nu,\nu_\pm}$ is a basis of $M$ by Corollary \ref{cone-subdiv}.
Denote by $l(E(t))$ the tropical length of the bounded edge $E(t)$, and by $l'(E(t))$ its derivative with respect to $t$.
We call its the {\em length gradient} of $E(t)$.
\begin{proposition}\label{prop-lprime}
We have 
$$l'(E(t))=-\tr(R_-,R,R_+)\in\Z,$$
\end{proposition}
In particular, the length gradients of all edges of $\Phi(t)$ are integer.
\begin{proof}
Choose the basis of $N$ so that $E(t)$ is horizontal, and the bissectrice $B_+$ is a vertical ray going down. Then the slope of the line extending the bissectrice $B_-$ is $\tr(R_-,R,R_+)$. On the other hand, minus this slope is the derivative of the length of the horizontal interval between $B_-$ and $B_+$ when we move this interval down with the unit velocity.
\end{proof}
The length gradient $l'(E(t))$ coincides with the value of the canonical class $K_{S(t)}$ of the toric surface $S(t)$ associated to the polytope $\Phi(t)$ on the homology class $[E(t)]\in H_2(S(t))$ of the toric divisor corresponding to the side $E(t)$.

\begin{proposition}
We have
$$l'(E(t))=K_{S(t)}([E(t)]).$$
\end{proposition}
\begin{proof}
Recall that the Poincar\'e dual \footnote{The toric surface $S(t)$ has quotient singularities (by local action of a finite group), and thus we have the Poincar\'e duality over $\Q$. For general toric surfaces, the canonical class is defined only as a cohomological class with $\Q$-coefficients while its Poincar\'e dual is always a homology class with $\Z$-coefficients (represented by the union of boundary divisors). However, in our case we also have $K_{S(t)}\in H^2(S(t);\Z)$ since by Theorem \ref{thm-candom}, $S(t)$ only has $A_n$-type singularities.} to the canonical class $K_{S(t)}$ is represented by the union $D_{\dd\Phi(t)}$ of the boundary divisors of the toric surface.
Thus we may compute $K_{S(t)}([E(t)])$ as the ($\Q$-valued) intersection number 
of $D_{\dd\Phi(t)}$
and the toric divisor $D_{E}$ corresponding to $E(t)$,
i.e. the self-intersection $[D_E].[D_E]$ plus the intersection numbers of $D_E$ with the two adjacent toric divisors.
Denote the rays in $\mr$ dual to these adjacent divisors by $M_\pm$ and the primitive integer vectors in these directions by $\mu_\pm$.

By Remark \ref{angle-sintersect}, 
we have
$$[D_E].[D_E]=-\tr(M_-,R,M_+)$$
Suppose that the two endpoints of $E(t)$ correspond to canonical singularities of $\Phi(t)$ of determinant $n_+$ and $n_-$.
Then we have $n_\pm\nu_\pm=\nu-\mu_\pm,$
and thus 
\begin{equation}\label{eq-corr}
\tr(M_-,R,M_+)=\tr(R_-,R,R_+)-(\frac1{n_+}+\frac1{n_-}).
\end{equation}
In the same time,
the intersection number of $D_E$ with the adjacent toric divisors is $\frac1{n_+}+\frac1{n_-}$, and thus the right-hand side of \eqref{eq-corr} is equal to $-K_{S(t)}([E(t)])$. The left-hand side is $-l'(E(t))$ by Proposition \ref{prop-lprime}.
\end{proof}

\section{Particle trajectory interpretation}
Suppose that $\Phi$ is a polygonal domain with rational slopes. Locally, near its every vertex, $\Phi$ is a cone, and thus its caustic is locally a finite union of ray emanating from the vertex. 
Let us imagine that each caustic ray of weight $w$ is the trajectory of a particle of weight $w$, emitted from this vertex, and moving with the velocity
given by the primitive integer vector parallel to this ray.
Then some of these particles may start to collide.
If the sum of all the momenta (velocities times the weight) of the colliding particle is zero then they annihilate each other,
and disappear. 
Otherwise, they collide into a single new particle whose momentum is the sum of the momenta of the colliding particles, all according to the balancing condition which is seen as the conservation of momentum law. It moves with the velocity given by the primitive integer vector parallel to the combined momentum.

This particle model describes the evolution of vertices of polygonal domains $\Phi(t)$ under the tropical wave front propagation.
Namely, we have the following statement. Recall that the final time $t_\Phi$ is the supremal time $t>0$ such that $\Phi(t)$ is a domain with non-zero interior. At time $t_\Phi$ the wave front evolution stop and $\Phi(t)=\emptyset$ if $t>t_\Phi$.  
\begin{proposition}\label{part-interpr}
At a time $t<t_\Phi$ a particle in the process described above corresponds to a vertex of the polygonal domain $\Phi(t)$. The weight of the particle is $n+1$ for a vertex of the type $A_n$.
\end{proposition}
The proposition is a restatement of Corollary \ref{An-single}.

We call a particle collision at a time $t=t_\Phi$ {\em final} as at a time $t>t_\Phi$ we have $\Phi(t)=\emptyset$. We call collisions at time $t<t_\Phi$ {\em non-final} or {\em interim}. 

\begin{theorem}\label{thm_nonfinal}
There is only one type of interim particle collisions, it involves exactly two particles, one of weight 1, another of arbitrary weight $n\ge 1$, and results in a new weight 1 particle. 

For each $n$ there is a unique type of an interim collision (up to automorphisms of the lattice $N$ and translations in $\nr$), it is depicted on Figure \ref{fig_nonterminal}. 
\end{theorem}
\begin{figure}
\begin{tikzpicture}
\node at (0,1.2) {$n$};
\draw[-latex](0,1)--(0,0.5);
\draw(0,0.5)--(0,0);

\node at (-0.5,-1.75) {$1$};
\draw[-latex](-0.5,-1.5)--(-0.25,-0.75);
\draw(-0.25,-0.75)--(0,0);

\node at (1.2,0) {$1$};
\draw[-latex](0,0)--(0.5,0);
\draw(0.5,0)--(1,0);

\begin{scope}[xshift=180,yshift=-25]
\draw(0,0)--(0,1)--(-3,1)--(0,0);

\draw[<->](0.025,1.2)--(-3.025,1.2);
\node at (-1.5,1.4) {$n$};

\node at (1,-1) {$\bullet$};
\node at (1,0) {$\bullet$};
\node at (1,1) {$\bullet$};
\node at (1,2) {$\bullet$};

\node at (0,-1) {$\bullet$};
\node at (0,0) {$\bullet$};
\node at (0,1) {$\bullet$};
\node at (0,2) {$\bullet$};

\node at (-1,-1) {$\bullet$};
\node at (-1,0) {$\bullet$};
\node at (-1,1) {$\bullet$};
\node at (-1,2) {$\bullet$};

\node at (-3,-1) {$\bullet$};
\node at (-3,0) {$\bullet$};
\node at (-3,1) {$\bullet$};
\node at (-3,2) {$\bullet$};

\node at (-2,-1) {$\dots$};
\node at (-2,-0) {$\dots$};

\node at (-2,2) {$\dots$};

\node at (-4,-1) {$\bullet$};
\node at (-4,0) {$\bullet$};
\node at (-4,1) {$\bullet$};
\node at (-4,2) {$\bullet$};
\end{scope}
\end{tikzpicture}
\caption{An interim collision of a particle of weight 1 and a particle of weight $n$ in $\nr$ and the dual diagram in $\mr$.
\label{fig_nonterminal}}
\end{figure}
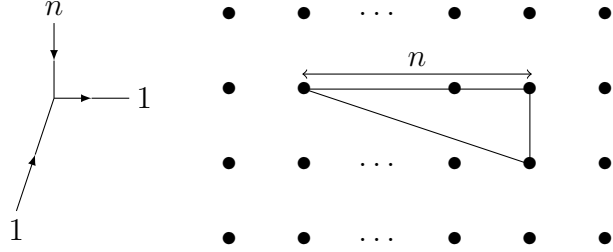
\begin{proof}
Consider a vertex $v$ of the polygonal domain $\Phi(t)$, where $0<t<t_\Phi$ is the time of the collision.
Suppose that some particles have collided at $v$. The collision results in a single particle by Corollary \ref{An-single} and Theorem  \ref{fan-description}.
The particles colliding at $v$ correspond to the vertices of $\Phi(t-\epsilon)$ near $v$ for a small $\epsilon>0$. The dual cones to these vertices subdivide the cone $\Sigma_v$ dual to $\Phi(t)$ at $v$. All of the cones in this subdivision are canonical.
Namely, the cone $\Sigma_v\subset\mr$ is subdivided by the rays spanned by the primitive vectors $r_1,\dots,r_{l-1}\in M$ so that each small cone is canonical. The cone $\Sigma_v$ is itself canonical, and is generated by the vectors $r_0,r_{l}\in M$ dual to the sides of $\Phi(t)$ adjacent to $v$. Furthermore, the subdivision must be {\em of convex type}, i.e. such that the points $O,r_0,\dots,r_l$ are the vertices of its convex hull, since the primitive vectors $r_j$ are dual to the sides of the polygonal domain $\Phi(t-\epsilon)$. 
The following lemma from tropical trigonometry completes the proof of the theorem. 
%
%
\end{proof}

We call a convex cone $\Sigma\in\mr$ {\em rational} if it is generated by two non-collinear vectors from $M$.
\begin{lemma}\label{trop-trig}
Suppose that a rational cone $\Sigma\subset\mr$ is subdivided into $l\ge 2$ $A_{n_j}$-cones, $j=1,\dots l$
so that the subdivision is of convex type. Then $\Sigma$ is of $A_0$-type (i.e. generated by two vectors from a basis of $M$), $l=2$ and either $n_1=1$ or $n_2=1$. Furthermore, such a subdivision is equivalent to the subdivision of the positive quadrant into two cones by the ray spanned by $(1,n)$ for some $n\in\N$.
\end{lemma}
\begin{proof}
Let $r_0,r_l\in M$ be the primitive vectors generating $\Sigma$, while $r_1,\dots,r_{l-1}\in M\cap\Sigma$ be the primitive vectors generating the subdivision. By the hypothesis, the convex hull $P$ of $O,r_0,\dots,r_l$ is a lattice $(l+2)$-gon, whose interior does not contain any lattice points.
In particular, the triangle with vertices $O,r_0, r_l$ does not contain any lattice points except for its vertices. Thus the vectors $r_0$ and $r_l$ form a basis, and therefore $\Sigma$ is of $A_0$-type.

By the same reasoning, any union of two or more adjacent cones in the subdivision is also an $A_0$-cone.
An $A_0$-cone is isomorphic to a positive quadrant in $\R^2$. Its decomposition into two rational cones is given by a single vector $(a,b)\in\Z^2$, $a,b>0$. If $b>a\ge 2$ then $(1,1)$ is in the interior of the triangle with vertices $O=(0,0)$, $(1,0)$ and $(a,b)$, and thus the lower cone is not of $A_n$-type. This implies the lemma if $l=2$.

Suppose that $l>2$. Identifying the cone $\Sigma$ with the positive quadrant in $\R^2$, consider the cone $\Sigma_x$ of the subdivision adjacent to the $x$-axis. By the first part of the proof, the union of all other cones in the subdivision is an $A_0$-cone. Thus $\Sigma_x$ is generated by $(1,0)$ and $(1,n)$ for some $n$. We get a contradiction since there are no lattice points 
in the positive quadrants such that the convex hull of $(0,0)$, $(1,0)$, $(1,n)$ and that point is a quadrilateral, as the $x$-coordinate of such point would have to be strictly between $0$ and $1$.
\end{proof}

\section{Final singularities of tropical caustics}
For an admissible domain $\Phi$, the final time $t_\Phi>0$ may be finite or infinite.
A useful criterion is provided by the {\em residual cone} of $\Phi$.
Let us recall its definition. 
Given a point $p\in\Phi$ we consider the cone $\Sigma(\Phi)$ obtained as the union of the rays $R$ in $\nr$ emanating from the origin, and such that $p+R\subset\Phi$.
By the next lemma, this cone does not depend on the choice of $p$, it is called the {\em residual cone} of $\Phi$.
Clearly, if $\Phi$ is bounded then $\Sigma(\Phi)=\emptyset$.
\begin{lemma}\label{res-indep}
Suppose that $p,q\in\Phi$ and $R\subset\nr$ is a ray centered at the origin.
We have $p+R\subset\Phi$ if and only if $q+R\subset\Phi$.
\end{lemma}
\begin{proof}
Suppose that we have $p+R\subset\Phi$. By convexity of $\Phi$, the union of closed intervals between $q$ and the points of $p+R\subset\Phi$ is contained in $\Phi$. The closure of this union contains $q+R$ and is contained in $\Phi$ since $\Phi$ is closed.
\end{proof}

It might happen that the residual cone $\Sigma(\Phi)\subset\nr$ consists of a finite number of rays.
Since $\Sigma(\Phi)$ is convex, three cases are possible: $\Sigma(\Phi)$ is empty (corresponding to the case when $\Phi$ is bounded), $\Sigma(\Phi)$ is a ray, and $\Sigma(\Phi)$ is a line.
All three cases are characterized by the property that $\Sigma(\Phi)$ is contained in a line in $\nr$.
\begin{proposition}\label{res-preserved}
If $\Sigma(\Phi)$ is not contained in a line then the residual cone is preserved under the tropical wave front evolution, i.e.
$\Sigma(\Phi(t))=\Sigma(\Phi)$, $t>0$.
\end{proposition}
\begin{proof}
Clearly, we have $\Sigma(\Phi(t))\subset\Sigma(\Phi)$.
Let us turn the tropical plane $\nr$ into the Euclidean plane by introducing a metric where an integer basis of $N$ becomes orthonormal. 
Then for each rational slope half-plane $H\subset\nr$ we can compare its tropical time $t$ evolution $H(t)$ against its Euclidean evolution $H^E(t)$ obtained by stepping out distance $t$ in the direction orthogonal to $\dd H$.
Since the length of any integer vector in $N$ is at least 1, we have 
$$H(t)\subset H^E(t),$$
and thus $\Sigma(\Phi(t))\supset\Sigma(\Phi)$.
\end{proof}

\begin{corollary}
If $\Sigma(\Phi)$ is not contained in a line then $t_\Phi=\infty$. 
\end{corollary}

If $\Phi$ is bounded then the caustic of $\Phi(\epsilon)$ is a finite bounded graph, $\epsilon>0$. Thus we get the following proposition.
\begin{proposition}
If $\Phi\subset\nr$ is bounded then $t_\Phi$ is finite.
\end{proposition}

\begin{remark}\label{rem-strip}
It is easy to see (cf. the proof of Lemma \ref{res-indep}) that if $\Sigma(\Phi)$ is a line then $\Phi$ is an infinite strip parallel to a line $L\subset\nr$ with a rational slope. In this case $t_\Phi$ is finite and $\Phi(t_\Phi)$ is a line parallel to $L$.

In the case when $\Sigma(\phi)$ is a single ray then we may have $t_\Phi<\infty$ (e.g. if $\Phi=[0,1]\times\R_{\ge 0}\subset\R^2$), but also $t=\infty$ (e.g. if $\Phi=\{(x,y)\in\R^2\ |\ y\ge x^2\}$).
\end{remark}

If $t_\Phi<\infty$ then we consider the locus $\Phi(t_\Phi)\subset\kf$.
We treat all points of $\Phi(t_\Phi)$ as singular points, and call then final singularities of $\kf$.

\begin{proposition}\label{finallocusclass}
If $\Phi\subset\nr$ is bounded then $\Phi(t_\Phi)$ is either a point or an interval with a rational slope. Otherwise, if $t_\Phi<\infty$ then $\Phi(t_\Phi)$ is either a line or a ray with a rational slope.
\end{proposition}
\begin{proof}
Since $\Phi(t_\Phi)$ is convex and with empty interior, it has to be a straight interval in $\nr$. Furthermore, it has to be contained in the caustic of a polygonal domain $\Phi(t_\Phi-\epsilon)$ for a small $\epsilon>0$. Thus this interval should be of rational slope.
The following lemma completes the proof of the proposition.
\end{proof}

\begin{lemma}\label{finalbounded}
The final locus $\Phi(t_\Phi)$ is bounded if and only if $\Phi$ is bounded.
\end{lemma}
\begin{proof}
Since $\Phi(t_\Phi)\subset\Phi$, the final locus must be bounded if the domain $\Phi$ is.
Suppose that $\Phi$ is unbounded, so that $\Sigma(\Phi)\neq\emptyset$. By Proposition \ref{res-preserved}, if $\Sigma(\Phi)$ is not contained in the line then $t_\Phi=\infty$. If $\Sigma(\Phi)$ is a line then $\Phi$ is a strip and $\Phi(t_\Phi)$ is a line by Remark \ref{rem-strip}. It remains to consider the case when $\Sigma(\Phi)$ consists of a single ray $R$.
Suppose that $p\in\Phi(t_\Phi)$. Then for any support half-plane $H$ to $\Phi$ we have $H\supset p+R$ and $H(t_\Phi)\supset p+R$. Thus $\Phi(t_\Phi)\supset p+R$. 
\end{proof}

\begin{proposition}\label{prop-point}
If $\Phi(t_\Phi)$ consists of a single point then this point is a vertex of $\kf$ locally isomorphic to the caustics of one of the 16 lattice polygons with a single lattice point in their interiors. These, so-called {\em reflexive polygons}, are depicted on Fig. \ref{fig_reflexivepoly}, and their caustics concide with the unions of the intervals from the interior point to the vertices (taken with the appropriate weights).
\end{proposition}
\begin{proof}
By Lemma \ref{finalbounded}, for a small $\epsilon>0$, the domain $\Phi(t_\Phi-\epsilon)$ is a finite polygon. It is canonical by Theorem \ref{fan-description}, and thus the polygon obtained as the convex hull of the primitive vectors in the direction of the dual fan is reflexive. The proposition now follows from the observation that the polygon dual to the reflexive polygon is also reflexive.
\end{proof}

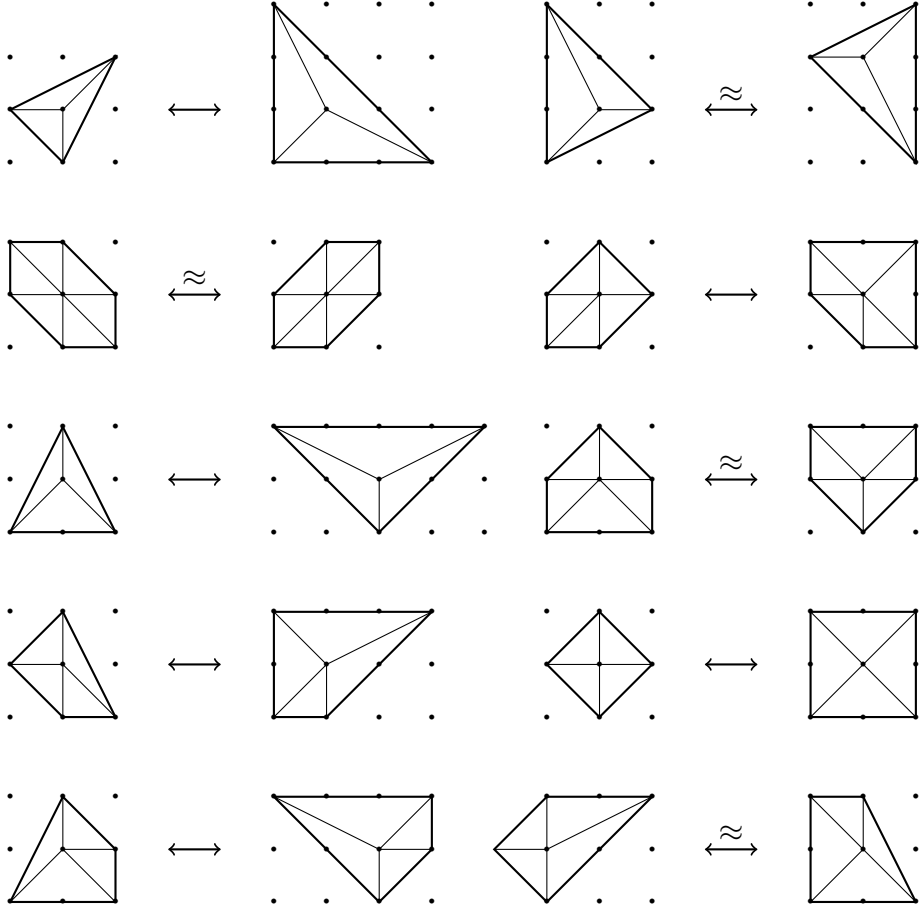
\begin{figure}
\begin{tikzpicture}

\begin{scope}[scale=0.7]

\node at (0,0) {$\boldsymbol{\cdot}$};
\node at (1,0) {$\boldsymbol{\cdot}$};
\node at (2,0) {$\boldsymbol{\cdot}$};
\node at (0,1) {$\boldsymbol{\cdot}$};
\node at (1,1) {$\boldsymbol{\cdot}$};
\node at (2,1) {$\boldsymbol{\cdot}$};
\node at (0,2) {$\boldsymbol{\cdot}$};
\node at (1,2) {$\boldsymbol{\cdot}$};
\node at (2,2) {$\boldsymbol{\cdot}$};

\draw[thick](1,0)--(2,2)--(0,1)--cycle;

\draw(1,1)--(1,0);
\draw(1,1)--(2,2);
\draw(1,1)--(0,1);

\draw[thick,<->] (3,1)--(4,1);

\begin{scope}[xshift=142.5]
\node at (0,0) {$\boldsymbol{\cdot}$};
\node at (1,0) {$\boldsymbol{\cdot}$};
\node at (2,0) {$\boldsymbol{\cdot}$};
\node at (3,0) {$\boldsymbol{\cdot}$};
\node at (0,1) {$\boldsymbol{\cdot}$};
\node at (1,1) {$\boldsymbol{\cdot}$};
\node at (2,1) {$\boldsymbol{\cdot}$};
\node at (3,1) {$\boldsymbol{\cdot}$};
\node at (0,2) {$\boldsymbol{\cdot}$};
\node at (1,2) {$\boldsymbol{\cdot}$};
\node at (2,2) {$\boldsymbol{\cdot}$};
\node at (3,2) {$\boldsymbol{\cdot}$};
\node at (0,3) {$\boldsymbol{\cdot}$};
\node at (1,3) {$\boldsymbol{\cdot}$};
\node at (2,3) {$\boldsymbol{\cdot}$};
\node at (3,3) {$\boldsymbol{\cdot}$};

\draw[thick](0,0)--(0,3)--(3,0)--cycle;

\draw(1,1)--(0,0);
\draw(1,1)--(3,0);
\draw(1,1)--(0,3);
\end{scope} 

\begin{scope}[xshift=290]
\node at (0,0) {$\boldsymbol{\cdot}$};
\node at (1,0) {$\boldsymbol{\cdot}$};
\node at (2,0) {$\boldsymbol{\cdot}$};
\node at (0,1) {$\boldsymbol{\cdot}$};
\node at (1,1) {$\boldsymbol{\cdot}$};
\node at (2,1) {$\boldsymbol{\cdot}$};
\node at (0,2) {$\boldsymbol{\cdot}$};
\node at (1,2) {$\boldsymbol{\cdot}$};
\node at (2,2) {$\boldsymbol{\cdot}$};
\node at (0,3) {$\boldsymbol{\cdot}$};
\node at (1,3) {$\boldsymbol{\cdot}$};
\node at (2,3) {$\boldsymbol{\cdot}$};

\draw[thick](0,0)--(0,3)--(2,1)--cycle;

\draw(1,1)--(0,0);
\draw(1,1)--(0,3);
\draw(1,1)--(2,1);

\draw[thick,<->] (3,1)--(4,1);

\node at (3.5,1.3) {$\approx$};

\begin{scope}[xshift=142.5]
\node at (0,0) {$\boldsymbol{\cdot}$};
\node at (1,0) {$\boldsymbol{\cdot}$};
\node at (2,0) {$\boldsymbol{\cdot}$};
\node at (0,1) {$\boldsymbol{\cdot}$};
\node at (1,1) {$\boldsymbol{\cdot}$};
\node at (2,1) {$\boldsymbol{\cdot}$};
\node at (0,2) {$\boldsymbol{\cdot}$};
\node at (1,2) {$\boldsymbol{\cdot}$};
\node at (2,2) {$\boldsymbol{\cdot}$};
\node at (0,3) {$\boldsymbol{\cdot}$};
\node at (1,3) {$\boldsymbol{\cdot}$};
\node at (2,3) {$\boldsymbol{\cdot}$};

\draw[thick](2,0)--(2,3)--(0,2)--cycle;

\draw(1,2)--(2,0);
\draw(1,2)--(2,3);
\draw(1,2)--(0,2);
\end{scope}

\end{scope}

\begin{scope}[yshift=-100]

\node at (0,0) {$\boldsymbol{\cdot}$};
\node at (1,0) {$\boldsymbol{\cdot}$};
\node at (2,0) {$\boldsymbol{\cdot}$};
\node at (0,1) {$\boldsymbol{\cdot}$};
\node at (1,1) {$\boldsymbol{\cdot}$};
\node at (2,1) {$\boldsymbol{\cdot}$};
\node at (0,2) {$\boldsymbol{\cdot}$};
\node at (1,2) {$\boldsymbol{\cdot}$};
\node at (2,2) {$\boldsymbol{\cdot}$};

\draw[thick](0,1)--(0,2)--(1,2)--(2,1)--(2,0)--(1,0)--cycle;

\draw(1,1)--(0,1);
\draw(1,1)--(0,2);
\draw(1,1)--(1,2);
\draw(1,1)--(2,1);
\draw(1,1)--(2,0);
\draw(1,1)--(1,0);

\draw[thick,<->] (3,1)--(4,1);

\node at (3.5,1.3) {$\approx$};

\begin{scope}[xshift=142.5]
\node at (0,0) {$\boldsymbol{\cdot}$};
\node at (1,0) {$\boldsymbol{\cdot}$};
\node at (2,0) {$\boldsymbol{\cdot}$};
\node at (0,1) {$\boldsymbol{\cdot}$};
\node at (1,1) {$\boldsymbol{\cdot}$};
\node at (2,1) {$\boldsymbol{\cdot}$};
\node at (0,2) {$\boldsymbol{\cdot}$};
\node at (1,2) {$\boldsymbol{\cdot}$};
\node at (2,2) {$\boldsymbol{\cdot}$};

\draw[thick](0,0)--(0,1)--(1,2)--(2,2)--(2,1)--(1,0)--cycle;

\draw(1,1)--(0,0);
\draw(1,1)--(0,1);
\draw(1,1)--(1,2);
\draw(1,1)--(2,2);
\draw(1,1)--(2,1);
\draw(1,1)--(1,0);

\end{scope} 

\begin{scope}[xshift=290]
\node at (0,0) {$\boldsymbol{\cdot}$};
\node at (1,0) {$\boldsymbol{\cdot}$};
\node at (2,0) {$\boldsymbol{\cdot}$};
\node at (0,1) {$\boldsymbol{\cdot}$};
\node at (1,1) {$\boldsymbol{\cdot}$};
\node at (2,1) {$\boldsymbol{\cdot}$};
\node at (0,2) {$\boldsymbol{\cdot}$};
\node at (1,2) {$\boldsymbol{\cdot}$};
\node at (2,2) {$\boldsymbol{\cdot}$};

\draw[thick](0,0)--(0,1)--(1,2)--(2,1)--(1,0)--cycle;

\draw(1,1)--(0,0);
\draw(1,1)--(0,1);
\draw(1,1)--(1,2);
\draw(1,1)--(2,1);
\draw(1,1)--(1,0);

\draw[thick,<->] (3,1)--(4,1);

\begin{scope}[xshift=142.5]
\node at (0,0) {$\boldsymbol{\cdot}$};
\node at (1,0) {$\boldsymbol{\cdot}$};
\node at (2,0) {$\boldsymbol{\cdot}$};
\node at (0,1) {$\boldsymbol{\cdot}$};
\node at (1,1) {$\boldsymbol{\cdot}$};
\node at (2,1) {$\boldsymbol{\cdot}$};
\node at (0,2) {$\boldsymbol{\cdot}$};
\node at (1,2) {$\boldsymbol{\cdot}$};
\node at (2,2) {$\boldsymbol{\cdot}$};

\draw[thick](0,1)--(0,2)--(2,2)--(2,0)--(1,0)--cycle;

\draw(1,1)--(0,1);
\draw(1,1)--(0,2);
\draw(1,1)--(2,2);
\draw(1,1)--(2,0);
\draw(1,1)--(1,0);

\end{scope}

\end{scope}

\end{scope}

\begin{scope}[yshift=-200]

\node at (0,0) {$\boldsymbol{\cdot}$};
\node at (1,0) {$\boldsymbol{\cdot}$};
\node at (2,0) {$\boldsymbol{\cdot}$};
\node at (0,1) {$\boldsymbol{\cdot}$};
\node at (1,1) {$\boldsymbol{\cdot}$};
\node at (2,1) {$\boldsymbol{\cdot}$};
\node at (0,2) {$\boldsymbol{\cdot}$};
\node at (1,2) {$\boldsymbol{\cdot}$};
\node at (2,2) {$\boldsymbol{\cdot}$};

\draw[thick](0,0)--(1,2)--(2,0)--cycle;

\draw(1,1)--(0,0);;
\draw(1,1)--(1,2);
\draw(1,1)--(2,0);

\draw[thick,<->] (3,1)--(4,1);

\begin{scope}[xshift=142.5]
\node at (0,0) {$\boldsymbol{\cdot}$};
\node at (1,0) {$\boldsymbol{\cdot}$};
\node at (2,0) {$\boldsymbol{\cdot}$};
\node at (3,0) {$\boldsymbol{\cdot}$};
\node at (4,0) {$\boldsymbol{\cdot}$};
\node at (0,1) {$\boldsymbol{\cdot}$};
\node at (1,1) {$\boldsymbol{\cdot}$};
\node at (2,1) {$\boldsymbol{\cdot}$};
\node at (3,1) {$\boldsymbol{\cdot}$};
\node at (4,1) {$\boldsymbol{\cdot}$};
\node at (0,2) {$\boldsymbol{\cdot}$};
\node at (1,2) {$\boldsymbol{\cdot}$};
\node at (2,2) {$\boldsymbol{\cdot}$};
\node at (3,2) {$\boldsymbol{\cdot}$};
\node at (4,2) {$\boldsymbol{\cdot}$};

\draw[thick](0,2)--(4,2)--(2,0)--cycle;

\draw(2,1)--(0,2);
\draw(2,1)--(4,2);
\draw(2,1)--(2,0);

\end{scope} 

\begin{scope}[xshift=290]
\node at (0,0) {$\boldsymbol{\cdot}$};
\node at (1,0) {$\boldsymbol{\cdot}$};
\node at (2,0) {$\boldsymbol{\cdot}$};
\node at (0,1) {$\boldsymbol{\cdot}$};
\node at (1,1) {$\boldsymbol{\cdot}$};
\node at (2,1) {$\boldsymbol{\cdot}$};
\node at (0,2) {$\boldsymbol{\cdot}$};
\node at (1,2) {$\boldsymbol{\cdot}$};
\node at (2,2) {$\boldsymbol{\cdot}$};

\draw[thick](0,0)--(0,1)--(1,2)--(2,1)--(2,0)--cycle;

\draw(1,1)--(0,0);
\draw(1,1)--(0,1);
\draw(1,1)--(1,2);
\draw(1,1)--(2,1);
\draw(1,1)--(2,0);

\draw[thick,<->] (3,1)--(4,1);

\node at (3.5,1.3) {$\approx$};

\begin{scope}[xshift=142.5]
\node at (0,0) {$\boldsymbol{\cdot}$};
\node at (1,0) {$\boldsymbol{\cdot}$};
\node at (2,0) {$\boldsymbol{\cdot}$};
\node at (0,1) {$\boldsymbol{\cdot}$};
\node at (1,1) {$\boldsymbol{\cdot}$};
\node at (2,1) {$\boldsymbol{\cdot}$};
\node at (0,2) {$\boldsymbol{\cdot}$};
\node at (1,2) {$\boldsymbol{\cdot}$};
\node at (2,2) {$\boldsymbol{\cdot}$};

\draw[thick](0,1)--(0,2)--(2,2)--(2,1)--(1,0)--cycle;

\draw(1,1)--(0,1);
\draw(1,1)--(0,2);
\draw(1,1)--(2,2);
\draw(1,1)--(2,1);
\draw(1,1)--(1,0);

\end{scope}

\end{scope}

\end{scope}

\begin{scope}[yshift=-300]

\node at (0,0) {$\boldsymbol{\cdot}$};
\node at (1,0) {$\boldsymbol{\cdot}$};
\node at (2,0) {$\boldsymbol{\cdot}$};
\node at (0,1) {$\boldsymbol{\cdot}$};
\node at (1,1) {$\boldsymbol{\cdot}$};
\node at (2,1) {$\boldsymbol{\cdot}$};
\node at (0,2) {$\boldsymbol{\cdot}$};
\node at (1,2) {$\boldsymbol{\cdot}$};
\node at (2,2) {$\boldsymbol{\cdot}$};

\draw[thick](0,1)--(1,2)--(2,0)--(1,0)--cycle;

\draw(1,1)--(0,1);
\draw(1,1)--(1,2);
\draw(1,1)--(2,0);
\draw(1,1)--(1,0);

\draw[thick,<->] (3,1)--(4,1);

\begin{scope}[xshift=142.5]
\node at (0,0) {$\boldsymbol{\cdot}$};
\node at (1,0) {$\boldsymbol{\cdot}$};
\node at (2,0) {$\boldsymbol{\cdot}$};
\node at (3,0) {$\boldsymbol{\cdot}$};
\node at (0,1) {$\boldsymbol{\cdot}$};
\node at (1,1) {$\boldsymbol{\cdot}$};
\node at (2,1) {$\boldsymbol{\cdot}$};
\node at (3,1) {$\boldsymbol{\cdot}$};
\node at (0,2) {$\boldsymbol{\cdot}$};
\node at (1,2) {$\boldsymbol{\cdot}$};
\node at (2,2) {$\boldsymbol{\cdot}$};
\node at (3,2) {$\boldsymbol{\cdot}$};

\draw[thick](0,0)--(0,2)--(3,2)--(1,0)--cycle;

\draw(1,1)--(0,0);
\draw(1,1)--(0,2);
\draw(1,1)--(3,2);
\draw(1,1)--(1,0);

\end{scope} 

\begin{scope}[xshift=290]
\node at (0,0) {$\boldsymbol{\cdot}$};
\node at (1,0) {$\boldsymbol{\cdot}$};
\node at (2,0) {$\boldsymbol{\cdot}$};
\node at (0,1) {$\boldsymbol{\cdot}$};
\node at (1,1) {$\boldsymbol{\cdot}$};
\node at (2,1) {$\boldsymbol{\cdot}$};
\node at (0,2) {$\boldsymbol{\cdot}$};
\node at (1,2) {$\boldsymbol{\cdot}$};
\node at (2,2) {$\boldsymbol{\cdot}$};

\draw[thick](0,1)--(1,2)--(2,1)--(1,0)--cycle;

\draw(1,1)--(1,0);
\draw(1,1)--(0,1);
\draw(1,1)--(2,1);
\draw(1,1)--(1,2);

\draw[thick,<->] (3,1)--(4,1);

\begin{scope}[xshift=142.5]
\node at (0,0) {$\boldsymbol{\cdot}$};
\node at (1,0) {$\boldsymbol{\cdot}$};
\node at (2,0) {$\boldsymbol{\cdot}$};
\node at (0,1) {$\boldsymbol{\cdot}$};
\node at (1,1) {$\boldsymbol{\cdot}$};
\node at (2,1) {$\boldsymbol{\cdot}$};
\node at (0,2) {$\boldsymbol{\cdot}$};
\node at (1,2) {$\boldsymbol{\cdot}$};
\node at (2,2) {$\boldsymbol{\cdot}$};

\draw[thick](0,0)--(0,2)--(2,2)--(2,0)--cycle;

\draw(1,1)--(0,0);
\draw(1,1)--(2,0);
\draw(1,1)--(0,2);
\draw(1,1)--(2,2);
\end{scope}

\end{scope}

\end{scope}

\begin{scope}[yshift=-400]

\node at (0,0) {$\boldsymbol{\cdot}$};
\node at (1,0) {$\boldsymbol{\cdot}$};
\node at (2,0) {$\boldsymbol{\cdot}$};
\node at (0,1) {$\boldsymbol{\cdot}$};
\node at (1,1) {$\boldsymbol{\cdot}$};
\node at (2,1) {$\boldsymbol{\cdot}$};
\node at (0,2) {$\boldsymbol{\cdot}$};
\node at (1,2) {$\boldsymbol{\cdot}$};
\node at (2,2) {$\boldsymbol{\cdot}$};

\draw[thick](0,0)--(1,2)--(2,1)--(2,0)--cycle;

\draw(1,1)--(0,0);;
\draw(1,1)--(1,2);
\draw(1,1)--(2,1);
\draw(1,1)--(2,0);

\draw[thick,<->] (3,1)--(4,1);

\begin{scope}[xshift=142.5]
\node at (0,0) {$\boldsymbol{\cdot}$};
\node at (1,0) {$\boldsymbol{\cdot}$};
\node at (2,0) {$\boldsymbol{\cdot}$};
\node at (3,0) {$\boldsymbol{\cdot}$};
\node at (0,1) {$\boldsymbol{\cdot}$};
\node at (1,1) {$\boldsymbol{\cdot}$};
\node at (2,1) {$\boldsymbol{\cdot}$};
\node at (3,1) {$\boldsymbol{\cdot}$};
\node at (0,2) {$\boldsymbol{\cdot}$};
\node at (1,2) {$\boldsymbol{\cdot}$};
\node at (2,2) {$\boldsymbol{\cdot}$};
\node at (3,2) {$\boldsymbol{\cdot}$};

\draw[thick](0,2)--(3,2)--(3,1)--(2,0)--cycle;

\draw(2,1)--(0,2);
\draw(2,1)--(3,2);
\draw(2,1)--(3,1);
\draw(2,1)--(2,0);

\end{scope} 

\begin{scope}[xshift=290]
\node at (0,0) {$\boldsymbol{\cdot}$};
\node at (1,0) {$\boldsymbol{\cdot}$};
\node at (2,0) {$\boldsymbol{\cdot}$};
\node at (0,1) {$\boldsymbol{\cdot}$};
\node at (1,1) {$\boldsymbol{\cdot}$};
\node at (2,1) {$\boldsymbol{\cdot}$};
\node at (0,2) {$\boldsymbol{\cdot}$};
\node at (1,2) {$\boldsymbol{\cdot}$};
\node at (2,2) {$\boldsymbol{\cdot}$};

\draw[thick](-1,1)--(0,2)--(2,2)--(0,0)--cycle;

\draw(0,1)--(-1,1);
\draw(0,1)--(0,2);
\draw(0,1)--(2,2);
\draw(0,1)--(0,0);

\draw[thick,<->] (3,1)--(4,1);

\node at (3.5,1.3) {$\approx$};

\begin{scope}[xshift=142.5]
\node at (0,0) {$\boldsymbol{\cdot}$};
\node at (1,0) {$\boldsymbol{\cdot}$};
\node at (2,0) {$\boldsymbol{\cdot}$};
\node at (0,1) {$\boldsymbol{\cdot}$};
\node at (1,1) {$\boldsymbol{\cdot}$};
\node at (2,1) {$\boldsymbol{\cdot}$};
\node at (0,2) {$\boldsymbol{\cdot}$};
\node at (1,2) {$\boldsymbol{\cdot}$};
\node at (2,2) {$\boldsymbol{\cdot}$};

\draw[thick](0,0)--(0,2)--(1,2)--(2,0)--cycle;

\draw(1,1)--(0,0);
\draw(1,1)--(0,2);
\draw(1,1)--(1,2);
\draw(1,1)--(2,0);

\end{scope}

\end{scope}

\end{scope}

\end{scope}
\end{tikzpicture}
\caption{Sixteen types of lattice polygons (up to automorphisms of the lattice) with a single lattice point in the interior together with their tropical caustics. Each caustic consists of segments connecting the central point to the vertices. The polygons are paired by duality, the multiplicity of an edge of the caustic is the length of the corresponding side of the dual polygon. Note that some of these types are self-dual.
}
\label{fig_reflexivepoly}
\end{figure}

\begin{proposition}\label{prop-edge}
If $I=\Phi(t_\Phi)$ consists of an interval or ray then $\kf$ has no vertices at the relative interior of $I$, it is an edge of $\kf$ of weight 2.
Near each endpoint of $I$, the caustic $\kf$ is locally isomorphic to a curve depicted on Fig. \ref{fig_finaledge}.
\end{proposition}
\begin{proof}
Note that no time $t_\Phi$ evolution $H(t_\Phi)$ of a support half-plane to $\Phi$ can intersect the relative interior of $I$, unless $\dd H$ is parallel to $I$.
Thus, $\kf$ only has vertices at the endpoints of $I$.
Consider the polygon in $\mr$ dual to such a vertex $v$.
Near $v$ we have a subdivision of the half-plane in $\mr$ (dual to $I$) into dual canonical cones (dual to the trajectories of colliding particles), cf. the proof of Theorem \ref{thm_nonfinal}.

Suppose that only two particles collide at $v$. Then we have a subdivision of a half-plane into an $A_n$-angle and its complimentary angle that has to be also of $A_n$-type. But this is possible only when $n=0$ or $n=1$ (the first two collisions of Fig. \ref{fig_finaledge}).

Suppose that more than two particles collide at $v$. Then the half-plane is subdivided into at least three angles of $A_n$-type, $\Sigma_1,\dots,\Sigma_k$ (ordered clockwise). 
By Lemma \ref{trop-trig} applied to the union of all but the first angle, we have $k=3$ while the union $\Sigma_2\cup\Sigma_3$ is an $A_0$ angle. This implies that $\Sigma_1$ is an $A_0$-angle itself since it complements an $A_0$ angle to the half-plane. By symmetry, the last angle $\Sigma_3$ is also an $A_0$-angle. Therefore, the resulting subdivision is isomorphic to the one depicted at the right of Figure \ref{fig_finaledge}.
\end{proof}


\begin{figure}
\begin{tikzpicture}
\node at (-1.15,1) {$1$};
\draw[-latex](-1,1)--(-0.5,0.5);
\draw(-0.5,0.5)--(0,0);

\node at (-1.15,-1.1) {$1$};
\draw[-latex](-1,-1)--(-0.5,-0.5);
\draw(-0.5,-0.5)--(0,0);

\node at (1.15,0) {$2$};
\draw[-latex,very thick](0,0)--(0.5,0);
\draw[very thick](0.5,0)--(1,0);

\begin{scope}[yshift=-90]
\draw(-1,0)--(0,1)--(0,-1)--(-1,0);
\node at (-1,1) {$\bullet$};
\node at (-1,0) {$\bullet$};
\node at (-1,-1) {$\bullet$};
\node at (0,1) {$\bullet$};
\node at (0,0) {$\bullet$};
\node at (0,-1) {$\bullet$};
\node at (1,1) {$\bullet$};
\node at (1,0) {$\bullet$};
\node at (1,-1) {$\bullet$};
\end{scope}

\begin{scope}[xshift=110]
\node at (0,1.2) {$2$};
\draw[-latex](0,1)--(0,0.5);
\draw(0,0.5)--(0,0);

\node at (-1.15,-1.1) {$2$};
\draw[-latex](-1,-1)--(-0.5,-0.5);
\draw(-0.5,-0.5)--(0,0);

\node at (1.15,0) {$2$};
\draw[-latex,very thick](0,0)--(0.5,0);
\draw[very thick](0.5,0)--(1,0);

\begin{scope}[yshift=-90]
\draw(-1,1)--(1,1)--(1,-1)--(-1,1);

\node at (-1,1) {$\bullet$};
\node at (-1,0) {$\bullet$};
\node at (-1,-1) {$\bullet$};
\node at (0,1) {$\bullet$};
\node at (0,0) {$\bullet$};
\node at (0,-1) {$\bullet$};
\node at (1,1) {$\bullet$};
\node at (1,0) {$\bullet$};
\node at (1,-1) {$\bullet$};
\end{scope}

\end{scope}

\begin{scope}[xshift=220]
\node at (0,1.2) {$1$};
\draw[-latex](0,1)--(0,0.5);
\draw(0,0.5)--(0,0);

\node at (-1.15,0) {$n$};
\draw[-latex](-1,0)--(-0.5,0);
\draw(-0.5,0)--(0,0);

\node at (1.15,-1.1) {$1$};
\draw[-latex](1,-1)--(0.5,-0.5);
\draw(0.5,-0.5)--(0,0);

\node at (1.15,0) {$2$};
\draw[-latex,very thick](0,0)--(0.5,0);
\draw[very thick](0.5,0)--(1,0);

\begin{scope}[yshift=-80,xshift=20]
\draw(0,-1)--(-1,-2)--(-1,1)--(0,1)--(0,-1);

\draw[<->](-1.2,1.025)--(-1.2,-2.025);
\node at (-1.3,-0.5) {$n$};

\node at (-2,-0.4) {$\vdots$};

\node at (-1,1) {$\bullet$};
\node at (-1,-2) {$\bullet$};

\node at (-2,1) {$\bullet$};
\node at (-2,-2) {$\bullet$};

\node at (0,1) {$\bullet$};
\node at (0,0) {$\bullet$};
\node at (0,-1) {$\bullet$};
\node at (0,-2) {$\bullet$};
\
\end{scope}

\end{scope}

\end{tikzpicture}
\caption{
Caustics near an endpoint of of $\Phi(t_\Phi)$ and their local dual polygons.
}\label{fig_finaledge}
\end{figure}
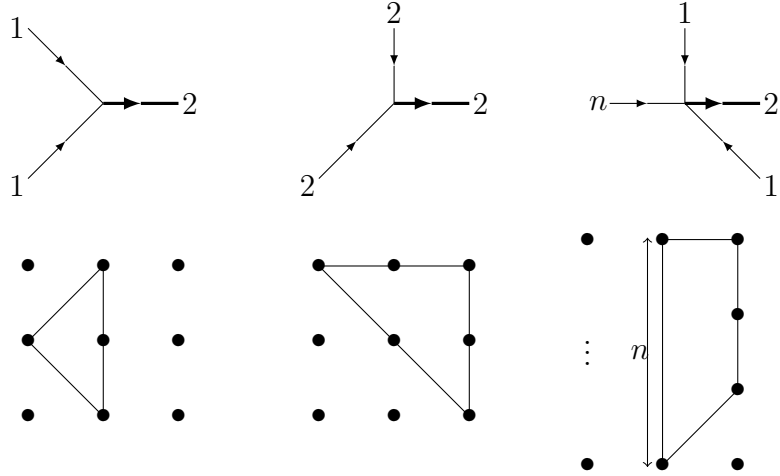

\section{Total length of a tropical caustic}
\newcommand{\Per}{\operatorname{l}}
Let $\Phi\subset N_\mathbb{R}$ be a compact convex domain. 
If $\Phi$ is polygonal (with rational slopes) then its caustic $\kf$ is a finite graph.
In this case we define $\Per(\kf)$ as the sum of tropical lengths of all edges of $\kf$ (multiplied by the corresponding weight).
We define $\Per(\dd\Phi)$ as the tropical perimeter of the polygon $\Phi$, i.e. the sum of tropical lengths of all the edges of $\dd\Phi$.
If $\Phi$ is not necessarily polygonal, we define $\Per(\kf)$ and $\Per(\dd\Phi)$ as the limits of the corresponding quantities for $\Phi(t)$ for $t>0$ when $t\to 0$.

The final locus $\Phi(t_\Phi)$ of a compact convex domain is either a finite interval or a point. Once again, we can measure its tropical length $\Per(\Phi(t_\Phi))$. This time we do not multiply by the weight. If $\Phi(t_\Phi)$ is a point then we set $\Per(\Phi(t_\Phi))=0$.

The formula \eqref{length-noether} was discovered in \cite{KaSh3}.
\begin{theorem} 
The limit values 
$\Per(\kf),\Per(\dd\Phi)<\infty$ exist for any compact convex domain $\Phi\subset\nr$.
Furthermore, we have
\begin{equation}
\label{length-noether}
\Per(\mathcal{K}_\Phi)+\Per(\partial\Phi)=12 t_\Phi+4\Per(\Phi(t_\Phi)).
\end{equation}
\label{prop_NoetherCaustic}
\end{theorem}
By Theorem \ref{fan-description}, if $\Phi(t')$, $t'>0$, has an edge of a given slope then $\Phi(t)$ also has an edge of the same slope whenever $0<t<t'$. Thus $\Per(\partial\Phi)$ can be computed as the sum of the tropical lengths of the edges of $\Phi$ of rational slope.
In particular, if $\dd\Phi$ does not contain intervals of rational slope then $\Per(\partial\Phi)=0$.
\begin{proof}
Consider the quantity
\begin{equation}\label{eq-qua}
\Per(\mathcal{K}_{\Phi(t)})+\Per(\partial\Phi(t))
\end{equation}
as a function of $t$, $0<t<1_\Phi$.
By Theorem \ref{12n}, its derivative 
 is $-12$.
Indeed, during time $\delta$ the caustic $\mathcal{K}_{\Phi(t)}$ gets shorter by $\delta d_v$ at every vertex $v$ while the length of each edge $E$ of $\partial\Phi(t)$
decreases by $\delta d_E$. When $t\to t_\Phi$, both $\mathcal{K}_{\Phi(t)}$ and $\partial\Phi(t)$ converge to $\Phi(t_\Phi)$, which is either a point, or an interval of weight 2. Thus
$
\lim\limits_{t\to t_\Phi} l(t)=4\Per(\Phi(t_\Phi)),
$
so that we get 
\begin{equation}\label{formula-t}
\Per(\mathcal{K}_{\Phi(t)})+\Per(\partial\Phi(t))=12(t_\Phi-t)+4\Per(\Phi(t_\Phi)).
\end{equation}
The length $\Per(\mathcal{K}_{\Phi(t)})$ is a decreasing function bounded by $12 t_\Phi+4\Per(\Phi(t_\Phi))$, since $\Per(\partial\Phi(t))\ge 0$. Therefore the limit $\Per(\kf)<\infty$ exists. By \eqref{formula-t}, the limit $\Per(\dd\Phi)<\infty$ exists as well.
\end{proof}

\section{Age of a polygonal domain}
If $\Phi$ is compact then we may use its caustic $\kf$ to encode the shape of the domain $\Phi$.
For arbitrary small $\epsilon>0$, $\kf\cap\Phi(\epsilon)$ is a finite tree, and thus has finitely many leaves (1-valent vertices). Its edges can be given a metric by setting the primitive integer vector to be of length 1 (independently of the edge weight). Note that this length agrees with the time needed for the wave front to pass from one point on the edge to another, unless the edge is the final locus.

Let us find two leaves of $\kf\cap\Phi(\epsilon)$ the most distant from each other in this metric, and denote the midpoint of the broken line connecticng these points. Then $p\in\Phi(t_\Phi)$. If $p$ is a vertex of $\kf$ then $\Phi(t_\Phi)=\{p\}$.
Otherwise, $\Phi(t_\Phi)$ is the edge of $\kf$ containing $p$ as its midpoint. Thus, we may reconstruct the final locus $\Phi(t_\Phi)$. To reconstruct $\Phi(t)$ for $0<t<t_\Phi$ we look at the points of $\kf$ at distance $t_\Phi-t$ from $\Phi(t_\Phi)$,
and connect the points adjacent to the same components of $\Phi\setminus\kf$ by straight intervals.

Note that the caustic $\kf$ has natural parameterisation: we choose the length of $\Phi(t_\Phi)$ (if it is an interval), then by 
Propositions \ref{prop-point} and \ref{prop-edge}, we have a discrete set of choices for the edges of $\kf$ adjacent to $\Phi(t_\Phi)$, all of these choices are realisable as caustics of the polygons $\Phi(t_\Phi-\epsilon)$ for small $\epsilon>0$.
When we increase $\epsilon$, any edge of weight 1 may split into two according to Theorem  \ref{thm_nonfinal}. Once any pair (or $k$-tuple for $k>2$) develop an intersection point, we must stop, and cannot go further. 

Let $\Phi$ be an admissible domain,
and $\kf$ be its caustic. If $\Phi$ is a canonical tropical domain then the closure $\overline\kf$ is a locally finite graph, while the vertices of $\Phi$ are the leaves (1-valent vertices) of $\overline\kf$.
Let $\widetilde\kf$ be the (balanced) graph obtained by extending indefinitely (as infinite rays going to infinity) the edges adjacent to the leaves of $\overline\kf$. Let ${\mathcal S}$ be the set of intersection points of the rays.

For each $a>0$ and each ray of $\widetilde\kf\setminus\kf$ we consider the interval of points of {\em tropical distance} not greater than $a$ from the endpoint of the ray \footnote{I.e. of Euclidean distance $a$ multiplied by the length of the primitive integer vector in the direction of the ray}. Let $\kf(-a)$ denote the union of $\kf$ and the set of all such points (on all rays) while
Let $\Phi(-a)$ be the convex hull of $\kf(-a)$.

\begin{proposition}\label{negative-evolution}
Let $\Phi$ be a canonical tropical domain.
If $\Phi(-a)\cap{\mathcal S}=\emptyset$ then $\Phi(-a)$ is also a canonical tropical domain whose sides are parallel to the sides of $\Phi$ and such that ${\mathcal K}_{\Phi(-a)}=\kf(-a)$.
\end{proposition}

\begin{proof}
By Proposition \ref{part-interpr}, the vertices of $\Phi$ move with the unit speed with respect to the tropical distance. By Corollary \ref{cone-subdiv}, the trajectories of the vertices adjacent to the same side both form the tropical right angle with that side. Therefore, the sides of $\Phi(-a)$ remain parallel to the sides of $\Phi$ unless the length of the side vanishes. The latter event corresponds to a point of ${\mathcal S}$.
\end{proof}

Definine $a_{\max}\in [0,+\infty]$ to be the infimal value of $a$ such that $\Phi(-a)\cap{\mathcal S}\neq\emptyset$. 
Proposition \ref{negative-evolution} implies the following statement.
\begin{corollary}\label{coro-age}
For any $0<t<a_{\max}$, the domain $\Phi(-t)$ is canonical. The result of its time $t$ evolution is $\Phi$, i.e. $(\Phi(-t))(t)=\Phi$.\end{corollary}

\newcommand{\age}{\operatorname{age}}
\begin{definition}
Let $\Phi\subset\nr$ be an admissible domain. The {\em age} of $\Phi$ is 
\begin{equation}\label{age-defn}
\age(\Phi)=\sup\{a\ge0\ |\ \exists\  
 \widetilde\Phi\subset\nr\ :\ \Phi=\widetilde\Phi(a)\},
\end{equation}
here $\widetilde\Phi\subset\nr$ is an admissible domain.
\end{definition}

\begin{proposition}
If $\Phi\subset\nr$ is a tropical canonical domain then $\age(\Phi)=a_{\max}$.
\end{proposition}
\begin{proof}
Consider a canonical tropical domain $\widetilde\Phi$ such that $\widetilde\Phi(t)=\Phi$, $t>0$.
Then, by the Huygens principle, we have ${\mathcal K}_{\tilde{\Phi}}\supset\kf$. If ${\mathcal K}_{\tilde{\Phi}}\setminus\kf$ consists of disjoint rays then $\widetilde\Phi=\Phi(-t)$, and thus $t\le a_{\max}$. The value $a_{\max}$ can be computed as the minimum of the ratio of the tropical length of sides to the rate of decreasing over all sides of $\Phi$. Each interim collision (described by Theorem \ref{thm_nonfinal} and viewed backwards in time) results in introducing a new side with an increasing length, but makes the adjacent sides to decrease faster. Thus it cannot increase the age of the polygon. 
%
%
\end{proof}

As $\Phi(0)=\Phi$, we have $\age(\Phi)\ge 0$. By Theorem \ref{fan-description}, $\age(\Phi)=0$ unless $\Phi$ is a canonical tropical domain. 

Conversely, if $\Phi$ is canonical and bounded then $\age(\Phi)>0$. 
Indeed, in this case the caustic $\kf$ is a finite graph while $\dd\Phi\cap\overline\kf$ consists of 1-valent vertices.
Thus $\Phi(-\epsilon)\cap{\mathcal S}=\emptyset$ for a small $\epsilon>0$, and by Corollary \ref{coro-age} we have $\age(\Phi)\ge\epsilon$.

\begin{remark}
There exist unbounded canonical tropical domains $\Phi$ of zero age, since the length of the sides might go to zero.
\end{remark}

\begin{remark}
By Theorem \ref{thm_nonfinal}, there are no interim collisions resulting in a particle of higher weight (greater than 1). Therefore, if $\Phi$ is a canonical tropical domain such that none of its vertex corresponds to the tropical right angle then it has the unique past, i.e. for any $\widetilde{\Phi}$ and $t>0$ such that $\widetilde{\Phi}(t)=\Phi$ we have $\widetilde{\Phi}=\Phi(-t)$.

In the same time if $\Phi$ is a canonical domain with at least one right angle vertex then its time $t$ past is not unique for any $0<t<\age(\Phi)$.
\end{remark}




\section{Examples}
Figures \ref{fig-triangle}, \ref{fig-5gon}, \ref{fig-9gon} represent the wave fronts at different time for a triangle, a pentagon and an enneagon.
The depicted polygons are not canonical tropical domains, their vertices yield several vertices of the wave front after eveloution in arbitrary small time.
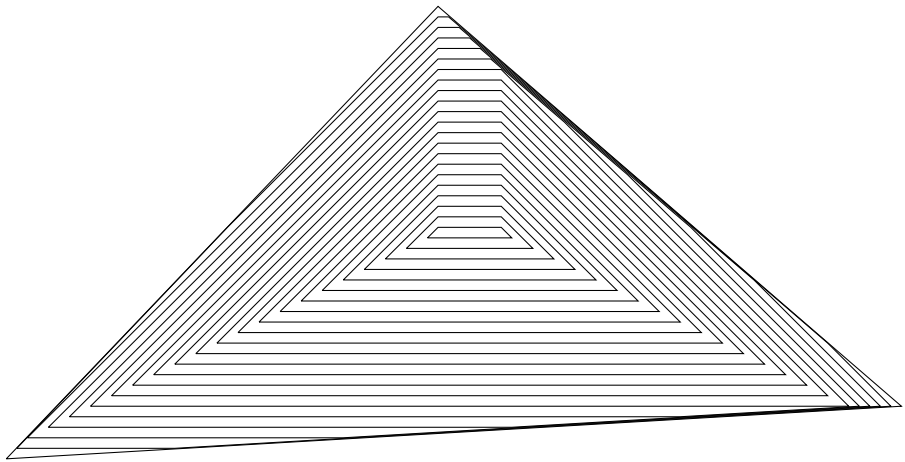
\begin{figure}
\begin{tikzpicture}
\begin{scope}[scale=0.14]
\draw(0,39)--(41,82)--(85,44)--(0,39);
\draw(1,40)--(21,61)--(41,81)--(42,81)--(56,69)--(78,50)--(84,44)--(16,40)--(1,40);
\draw(2,41)--(41,80)--(43,80)--(71,56)--(83,44)--(32,41)--(2,41);
\draw(4,42)--(41,79)--(44,79)--(65,61)--(82,44)--(48,42)--(4,42);
\draw(6,43)--(41,78)--(45,78)--(59,66)--(81,44)--(64,43)--(6,43);
\draw(8,44)--(41,77)--(46,77)--(53,71)--(80,44)--(8,44);
\draw(10,45)--(41,76)--(47,76)--(78,45)--(10,45);
\draw(12,46)--(41,75)--(47,75)--(76,46)--(12,46);
\draw(14,47)--(41,74)--(47,74)--(74,47)--(14,47);
\draw(16,48)--(41,73)--(47,73)--(72,48)--(16,48);
\draw(18,49)--(41,72)--(47,72)--(70,49)--(18,49);
\draw(20,50)--(41,71)--(47,71)--(68,50)--(20,50);
\draw(22,51)--(41,70)--(47,70)--(66,51)--(22,51);
\draw(24,52)--(41,69)--(47,69)--(64,52)--(24,52);
\draw(26,53)--(41,68)--(47,68)--(62,53)--(26,53);
\draw(28,54)--(41,67)--(47,67)--(60,54)--(28,54);
\draw(30,55)--(41,66)--(47,66)--(58,55)--(30,55);
\draw(32,56)--(41,65)--(47,65)--(56,56)--(32,56);
\draw(34,57)--(41,64)--(47,64)--(54,57)--(34,57);
\draw(36,58)--(41,63)--(47,63)--(52,58)--(36,58);
\draw(38,59)--(41,62)--(47,62)--(50,59)--(38,59);
\draw(40,60)--(41,61)--(47,61)--(48,60)--(40,60);
\end{scope}
\end{tikzpicture}
\caption{A polygonal tropical wave front of a lattice triangle.}
\label{fig-triangle}
\end{figure}
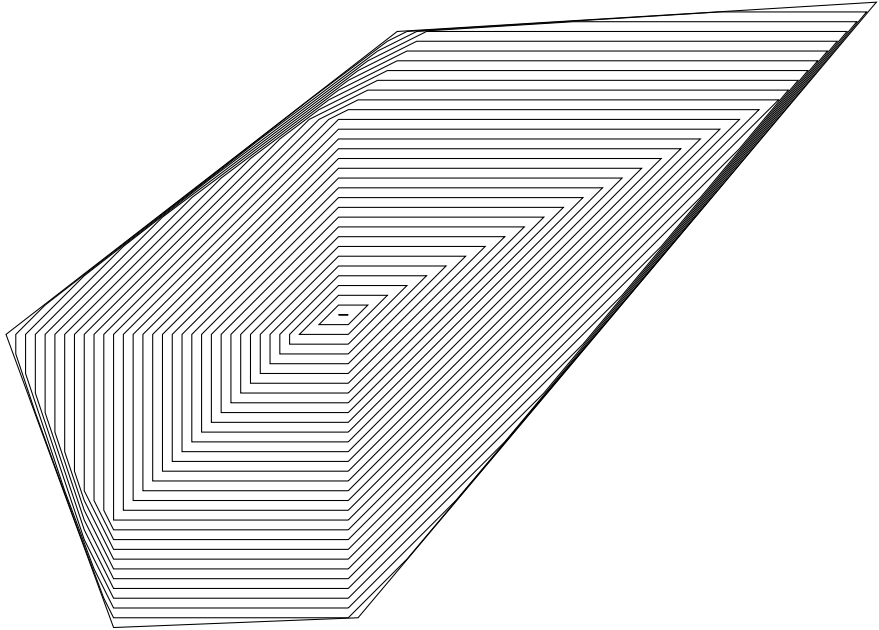
\begin{figure}
\rotatebox{90}{
\begin{tikzpicture}
\begin{scope}[scale=0.13]
\draw(31,82)--(61,93)--(92,53)--(95,4)--(32,57)--(31,82);
\draw(32,58)--(32,82)--(34,83)--(42,86)--(53,90)--(59,92)--(61,92)--(64,89)--(85,62)--(91,54)--(92,52)--(94,20)--(94,5)--(76,20)--(38,52)--(32,58);
\draw(33,58)--(33,82)--(37,84)--(45,87)--(57,91)--(61,91)--(67,85)--(81,67)--(90,55)--(92,51)--(93,35)--(93,6)--(63,31)--(44,47)--(33,58);
\draw(34,58)--(34,82)--(40,85)--(55,90)--(61,90)--(70,81)--(77,72)--(89,56)--(92,50)--(92,7)--(50,42)--(34,58);
\draw(35,58)--(35,82)--(41,85)--(53,89)--(61,89)--(73,77)--(88,57)--(91,51)--(91,8)--(55,38)--(35,58);
\draw(36,58)--(36,82)--(42,85)--(51,88)--(61,88)--(75,74)--(87,58)--(90,52)--(90,9)--(60,34)--(36,58);
\draw(37,58)--(37,82)--(43,85)--(49,87)--(61,87)--(77,71)--(86,59)--(89,53)--(89,10)--(65,30)--(37,58);
\draw(38,58)--(38,82)--(44,85)--(47,86)--(61,86)--(79,68)--(85,60)--(88,54)--(88,11)--(70,26)--(38,58);
\draw(39,58)--(39,82)--(45,85)--(61,85)--(81,65)--(84,61)--(87,55)--(87,12)--(75,22)--(39,58);
\draw(40,58)--(40,82)--(44,84)--(61,84)--(83,62)--(86,56)--(86,13)--(80,18)--(40,58);
\draw(41,58)--(41,82)--(43,83)--(61,83)--(83,61)--(85,57)--(85,14)--(41,58);
\draw(42,58)--(42,82)--(61,82)--(83,60)--(84,58)--(84,16)--(42,58);
\draw(43,58)--(43,81)--(61,81)--(83,59)--(83,18)--(43,58);
\draw(44,58)--(44,80)--(61,80)--(82,59)--(82,20)--(44,58);
\draw(45,58)--(45,79)--(61,79)--(81,59)--(81,22)--(45,58);
\draw(46,58)--(46,78)--(61,78)--(80,59)--(80,24)--(46,58);
\draw(47,58)--(47,77)--(61,77)--(79,59)--(79,26)--(47,58);
\draw(48,58)--(48,76)--(61,76)--(78,59)--(78,28)--(48,58);
\draw(49,58)--(49,75)--(61,75)--(77,59)--(77,30)--(49,58);
\draw(50,58)--(50,74)--(61,74)--(76,59)--(76,32)--(50,58);
\draw(51,58)--(51,73)--(61,73)--(75,59)--(75,34)--(51,58);
\draw(52,58)--(52,72)--(61,72)--(74,59)--(74,36)--(52,58);
\draw(53,58)--(53,71)--(61,71)--(73,59)--(73,38)--(53,58);
\draw(54,58)--(54,70)--(61,70)--(72,59)--(72,40)--(54,58);
\draw(55,58)--(55,69)--(61,69)--(71,59)--(71,42)--(55,58);
\draw(56,58)--(56,68)--(61,68)--(70,59)--(70,44)--(56,58);
\draw(57,58)--(57,67)--(61,67)--(69,59)--(69,46)--(57,58);
\draw(58,58)--(58,66)--(61,66)--(68,59)--(68,48)--(58,58);
\draw(59,58)--(59,65)--(61,65)--(67,59)--(67,50)--(59,58);
\draw(60,58)--(60,64)--(61,64)--(66,59)--(66,52)--(60,58);
\draw(61,58)--(61,63)--(65,59)--(65,54)--(61,58);
\draw(62,58)--(62,61)--(64,59)--(64,56)--(62,58);
\draw(63,58)--(63,59)--(63,58);
\end{scope}
\end{tikzpicture}}
\caption{A polygonal tropical wave front of a lattice pentagon.}
\label{fig-5gon}
\end{figure}
\begin{figure}
\begin{tikzpicture}
\begin{scope}[scale=0.09]
\draw(0,3)--(1,12)--(14,93)--(59,99)--(97,97)--(96,44)--(88,18)--(57,11)--(4,0)--(0,3);
\draw(1,3)--(1,11)--(2,18)--(10,68)--(14,92)--(15,93)--(22,94)--(52,98)--(77,98)--(96,97)--(96,45)--(95,41)--(91,28)--(88,19)--(87,18)--(79,16)--(52,10)--(28,5)--(8,1)--(3,1)--(1,3);
\draw(2,3)--(2,17)--(3,24)--(7,49)--(14,91)--(16,93)--(30,95)--(45,97)--(95,97)--(95,42)--(94,38)--(88,20)--(86,18)--(74,15)--(47,9)--(12,2)--(3,2)--(2,3);
\draw(3,3)--(3,23)--(4,30)--(14,90)--(17,93)--(38,96)--(94,96)--(94,39)--(88,21)--(85,18)--(69,14)--(51,10)--(16,3)--(3,3);
\draw(4,4)--(4,29)--(14,89)--(18,93)--(32,95)--(93,95)--(93,37)--(88,22)--(84,18)--(64,13)--(55,11)--(20,4)--(4,4);
\draw(5,5)--(5,34)--(14,88)--(19,93)--(26,94)--(92,94)--(92,35)--(88,23)--(83,18)--(59,12)--(24,5)--(5,5);
\draw(6,6)--(6,39)--(14,87)--(20,93)--(91,93)--(91,33)--(88,24)--(82,18)--(58,12)--(28,6)--(6,6);
\draw(7,7)--(7,44)--(14,86)--(20,92)--(90,92)--(90,31)--(88,25)--(81,18)--(57,12)--(32,7)--(7,7);
\draw(8,8)--(8,49)--(14,85)--(20,91)--(89,91)--(89,29)--(88,26)--(80,18)--(56,12)--(36,8)--(8,8);
\draw(9,9)--(9,54)--(14,84)--(20,90)--(88,90)--(88,27)--(79,18)--(55,12)--(40,9)--(9,9);
\draw(10,10)--(10,59)--(14,83)--(20,89)--(87,89)--(87,27)--(78,18)--(54,12)--(44,10)--(10,10);
\draw(11,11)--(11,64)--(14,82)--(20,88)--(86,88)--(86,27)--(77,18)--(53,12)--(48,11)--(11,11);
\draw(12,12)--(12,69)--(14,81)--(20,87)--(85,87)--(85,27)--(76,18)--(52,12)--(12,12);
\draw(13,13)--(13,74)--(14,80)--(20,86)--(84,86)--(84,27)--(75,18)--(55,13)--(13,13);
\draw(14,14)--(14,79)--(20,85)--(83,85)--(83,27)--(74,18)--(58,14)--(14,14);
\draw(15,15)--(15,79)--(20,84)--(82,84)--(82,27)--(73,18)--(61,15)--(15,15);
\draw(16,16)--(16,79)--(20,83)--(81,83)--(81,27)--(72,18)--(64,16)--(16,16);
\draw(17,17)--(17,79)--(20,82)--(80,82)--(80,27)--(71,18)--(67,17)--(17,17);
\draw(18,18)--(18,79)--(20,81)--(79,81)--(79,27)--(70,18)--(18,18);
\draw(19,19)--(19,79)--(20,80)--(78,80)--(78,27)--(70,19)--(19,19);
\draw(20,20)--(20,79)--(77,79)--(77,27)--(70,20)--(20,20);
\draw(21,21)--(21,78)--(76,78)--(76,27)--(70,21)--(21,21);
\draw(22,22)--(22,77)--(75,77)--(75,27)--(70,22)--(22,22);
\draw(23,23)--(23,76)--(74,76)--(74,27)--(70,23)--(23,23);
\draw(24,24)--(24,75)--(73,75)--(73,27)--(70,24)--(24,24);
\draw(25,25)--(25,74)--(72,74)--(72,27)--(70,25)--(25,25);
\draw(26,26)--(26,73)--(71,73)--(71,27)--(70,26)--(26,26);
\draw(27,27)--(27,72)--(70,72)--(70,27)--(27,27);
\draw(28,28)--(28,71)--(69,71)--(69,28)--(28,28);
\draw(29,29)--(29,70)--(68,70)--(68,29)--(29,29);
\draw(30,30)--(30,69)--(67,69)--(67,30)--(30,30);
\draw(31,31)--(31,68)--(66,68)--(66,31)--(31,31);
\draw(32,32)--(32,67)--(65,67)--(65,32)--(32,32);
\draw(33,33)--(33,66)--(64,66)--(64,33)--(33,33);
\draw(34,34)--(34,65)--(63,65)--(63,34)--(34,34);
\draw(35,35)--(35,64)--(62,64)--(62,35)--(35,35);
\draw(36,36)--(36,63)--(61,63)--(61,36)--(36,36);
\draw(37,37)--(37,62)--(60,62)--(60,37)--(37,37);
\draw(38,38)--(38,61)--(59,61)--(59,38)--(38,38);
\draw(39,39)--(39,60)--(58,60)--(58,39)--(39,39);
\draw(40,40)--(40,59)--(57,59)--(57,40)--(40,40);
\draw(41,41)--(41,58)--(56,58)--(56,41)--(41,41);
\draw(42,42)--(42,57)--(55,57)--(55,42)--(42,42);
\draw(43,43)--(43,56)--(54,56)--(54,43)--(43,43);
\draw(44,44)--(44,55)--(53,55)--(53,44)--(44,44);
\draw(45,45)--(45,54)--(52,54)--(52,45)--(45,45);
\draw(46,46)--(46,53)--(51,53)--(51,46)--(46,46);
\draw(47,47)--(47,52)--(50,52)--(50,47)--(47,47);
\draw(48,48)--(48,51)--(49,51)--(49,48)--(48,48);
\end{scope}
\end{tikzpicture}
\caption{A polygonal tropical wave front of a lattice enneagon.}
\label{fig-9gon}
\end{figure}
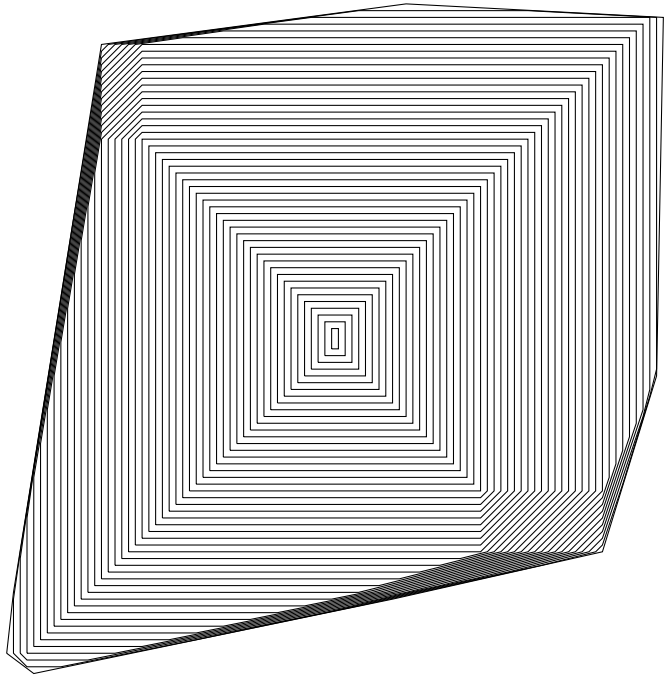

Figures \ref{fig-ellipse} and \ref{fig-ellipse-waves} depict the tropical caustic of an ellipse and the corresponding wave fronts for its evolution after a small time. Note that when the starting contour is smooth (e.g. an ellipse), the caustics has infinitely many branching points accumulating at the starting contour.
\input Fig-ellipse
\begin{figure}
\begin{tikzpicture}
\begin{scope}[scale=0.13]
\node at (0,36) {\ };
\draw(-49,-6)--(-49,6)--(-46,12)--(-38,20)--(-35,22)--(-27,26)--(-21,28)--(-17,29)--(-11,30)--(11,30)--(17,29)--(21,28)--(27,26)--(35,22)--(38,20)--(46,12)--(49,6)--(49,-6)--(46,-12)--(38,-20)--(35,-22)--(27,-26)--(21,-28)--(17,-29)--(11,-30)--(-11,-30)--(-17,-29)--(-21,-28)--(-27,-26)--(-35,-22)--(-38,-20)--(-46,-12)--(-49,-6);
\draw(-48,-7)--(-48,7)--(-46,11)--(-36,21)--(-26,26)--(-20,28)--(-16,29)--(16,29)--(20,28)--(26,26)--(36,21)--(46,11)--(48,7)--(48,-7)--(46,-11)--(36,-21)--(26,-26)--(20,-28)--(16,-29)--(-16,-29)--(-20,-28)--(-26,-26)--(-36,-21)--(-46,-11)--(-48,-7);
\draw(-47,-8)--(-47,8)--(-46,10)--(-35,21)--(-25,26)--(-19,28)--(19,28)--(25,26)--(35,21)--(46,10)--(47,8)--(47,-8)--(46,-10)--(35,-21)--(25,-26)--(19,-28)--(-19,-28)--(-25,-26)--(-35,-21)--(-46,-10)--(-47,-8);
\draw(-46,-9)--(-46,9)--(-34,21)--(-24,26)--(-21,27)--(21,27)--(24,26)--(34,21)--(46,9)--(46,-9)--(34,-21)--(24,-26)--(21,-27)--(-21,-27)--(-24,-26)--(-34,-21)--(-46,-9);
\draw(-45,-9)--(-45,9)--(-33,21)--(-23,26)--(23,26)--(33,21)--(45,9)--(45,-9)--(33,-21)--(23,-26)--(-23,-26)--(-33,-21)--(-45,-9);
\draw(-44,-9)--(-44,9)--(-32,21)--(-24,25)--(24,25)--(32,21)--(44,9)--(44,-9)--(32,-21)--(24,-25)--(-24,-25)--(-32,-21)--(-44,-9);
\draw(-43,-9)--(-43,9)--(-31,21)--(-25,24)--(25,24)--(31,21)--(43,9)--(43,-9)--(31,-21)--(25,-24)--(-25,-24)--(-31,-21)--(-43,-9);
\draw(-42,-9)--(-42,9)--(-30,21)--(-26,23)--(26,23)--(30,21)--(42,9)--(42,-9)--(30,-21)--(26,-23)--(-26,-23)--(-30,-21)--(-42,-9);
\draw(-41,-9)--(-41,9)--(-29,21)--(-27,22)--(27,22)--(29,21)--(41,9)--(41,-9)--(29,-21)--(27,-22)--(-27,-22)--(-29,-21)--(-41,-9);
\draw(-40,-9)--(-40,9)--(-28,21)--(28,21)--(40,9)--(40,-9)--(28,-21)--(-28,-21)--(-40,-9);
\draw(-39,-9)--(-39,9)--(-28,20)--(28,20)--(39,9)--(39,-9)--(28,-20)--(-28,-20)--(-39,-9);
\draw(-38,-9)--(-38,9)--(-28,19)--(28,19)--(38,9)--(38,-9)--(28,-19)--(-28,-19)--(-38,-9);
\draw(-37,-9)--(-37,9)--(-28,18)--(28,18)--(37,9)--(37,-9)--(28,-18)--(-28,-18)--(-37,-9);
\draw(-36,-9)--(-36,9)--(-28,17)--(28,17)--(36,9)--(36,-9)--(28,-17)--(-28,-17)--(-36,-9);
\draw(-35,-9)--(-35,9)--(-28,16)--(28,16)--(35,9)--(35,-9)--(28,-16)--(-28,-16)--(-35,-9);
\draw(-34,-9)--(-34,9)--(-28,15)--(28,15)--(34,9)--(34,-9)--(28,-15)--(-28,-15)--(-34,-9);
\draw(-33,-9)--(-33,9)--(-28,14)--(28,14)--(33,9)--(33,-9)--(28,-14)--(-28,-14)--(-33,-9);
\draw(-32,-9)--(-32,9)--(-28,13)--(28,13)--(32,9)--(32,-9)--(28,-13)--(-28,-13)--(-32,-9);
\draw(-31,-9)--(-31,9)--(-28,12)--(28,12)--(31,9)--(31,-9)--(28,-12)--(-28,-12)--(-31,-9);
\draw(-30,-9)--(-30,9)--(-28,11)--(28,11)--(30,9)--(30,-9)--(28,-11)--(-28,-11)--(-30,-9);
\draw(-29,-9)--(-29,9)--(-28,10)--(28,10)--(29,9)--(29,-9)--(28,-10)--(-28,-10)--(-29,-9);
\draw(-28,-9)--(-28,9)--(28,9)--(28,-9)--(-28,-9);
\draw(-27,-8)--(-27,8)--(27,8)--(27,-8)--(-27,-8);
\draw(-26,-7)--(-26,7)--(26,7)--(26,-7)--(-26,-7);
\draw(-25,-6)--(-25,6)--(25,6)--(25,-6)--(-25,-6);
\draw(-24,-5)--(-24,5)--(24,5)--(24,-5)--(-24,-5);
\draw(-23,-4)--(-23,4)--(23,4)--(23,-4)--(-23,-4);
\draw(-22,-3)--(-22,3)--(22,3)--(22,-3)--(-22,-3);
\draw(-21,-2)--(-21,2)--(21,2)--(21,-2)--(-21,-2);
\draw(-20,-1)--(-20,1)--(20,1)--(20,-1)--(-20,-1);
\draw(-19,0)--(19,0)--(-19,0);
\end{scope}
\end{tikzpicture}
\caption{Polygonal tropical wave front for a lattice approximation of a rescaling of the same ellipse.}
\label{fig-ellipse-waves}
\end{figure}
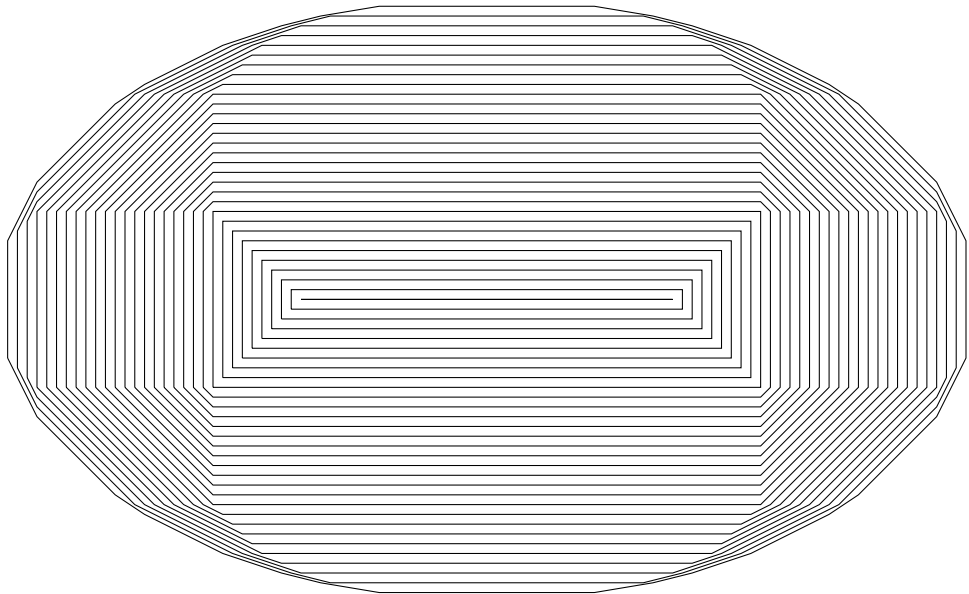

Figure \ref{fig-cubic} depicts the tropical caustic for the convex bounded domain in the complement of a nodal cubic curve.
The large vertical edge in the right half of the domain intersects the horizontal edge of weight 2. The horizontal edge to the right of this intersection point is the trajectory of the $A_1$-vertex of the wave front.
The horizontal edge to the left of this intersection point is the final locus.
\begin{figure}
\begin{tikzpicture}
\begin{scope}[scale=8]

\draw[ultra thick](0,0)--(-0.61509982054024948,0);

\draw[thick,scale=1, domain=-1.12:1.12, smooth, variable=\t] plot ({\t*\t-1}, {\t*(\t*\t-1)});

\draw(-0.748731263,0.372122852513)--(-0.734631914601,0.376822635313);
\draw(-0.748731263,0.372122852513)--(-0.759888532171,0.372122852513);
\draw(-0.769525139018,0.361725914504)--(-0.748731263,0.372122852513);
\draw(-0.769525139018,0.361725914504)--(-0.788796531558,0.361725914504);
\draw(-0.800285005725,0.330966047797)--(-0.769525139018,0.361725914504);
\draw(-0.840954285046,0.330966047797)--(-0.830807885577,0.341112447266);
\draw(-0.840954285046,0.330966047797)--(-0.855605328548,0.323640526046);
\draw(-0.800285005725,0.330966047797)--(-0.840954285046,0.330966047797);
\draw(-0.800285005725,0.185185185185)--(-0.800285005725,0.330966047797);
\draw(-0.918383099822,0.260161385432)--(-0.91241367943,0.269115516019);
\draw(-0.927426605737,0.242074373601)--(-0.918383099822,0.260161385432);
\draw(-0.927426605737,0.242074373601)--(-0.939146189893,0.230354789445);
\draw(-0.927426605737,0.185185185185)--(-0.927426605737,0.242074373601);
\draw(-0.966589644756,0.173956564089)--(-0.963333245098,0.183725763062);
\draw(-0.966589644756,0.146022146167)--(-0.966589644756,0.173956564089);
\draw(-0.981471151527,0.116259132623)--(-0.981471151527,0.132418221737);
\draw(-0.981471151527,0.116259132623)--(-0.988369888322,0.09556292224);
\draw(-0.966589644756,0.146022146167)--(-0.981471151527,0.116259132623);
\draw(-0.927426605737,0.185185185185)--(-0.966589644756,0.146022146167);
\draw(-0.800285005725,0.185185185185)--(-0.927426605737,0.185185185185);
\draw(-0.61509982054,0.)--(-0.800285005725,0.185185185185);
\draw(-0.981471151527,-0.116259132623)--(-0.988369888322,-0.09556292224);
\draw(-0.981471151527,-0.116259132623)--(-0.981471151527,-0.132418221737);
\draw(-0.966589644756,-0.146022146167)--(-0.981471151527,-0.116259132623);
\draw(-0.966589644756,-0.173956564089)--(-0.963333245098,-0.183725763062);
\draw(-0.966589644756,-0.146022146167)--(-0.966589644756,-0.173956564089);
\draw(-0.927426605737,-0.185185185185)--(-0.966589644756,-0.146022146167);
\draw(-0.927426605737,-0.242074373601)--(-0.939146189893,-0.230354789445);
\draw(-0.918383099822,-0.260161385432)--(-0.91241367943,-0.269115516019);
\draw(-0.927426605737,-0.242074373601)--(-0.918383099822,-0.260161385432);
\draw(-0.927426605737,-0.185185185185)--(-0.927426605737,-0.242074373601);
\draw(-0.800285005725,-0.185185185185)--(-0.927426605737,-0.185185185185);
\draw(-0.840954285046,-0.330966047797)--(-0.855605328548,-0.323640526046);
\draw(-0.840954285046,-0.330966047797)--(-0.830807885577,-0.341112447266);
\draw(-0.800285005725,-0.330966047797)--(-0.840954285046,-0.330966047797);
\draw(-0.769525139018,-0.361725914504)--(-0.788796531558,-0.361725914504);
\draw(-0.748731263,-0.372122852513)--(-0.759888532171,-0.372122852513);
\draw(-0.748731263,-0.372122852513)--(-0.734631914601,-0.376822635313);
\draw(-0.769525139018,-0.361725914504)--(-0.748731263,-0.372122852513);
\draw(-0.800285005725,-0.330966047797)--(-0.769525139018,-0.361725914504);
\draw(-0.800285005725,-0.185185185185)--(-0.800285005725,-0.330966047797);
\draw(-0.61509982054,0.)--(-0.800285005725,-0.185185185185);
\draw(-0.579442159473,-0.37348930198)--(-0.595588726937,-0.377525943846);
\draw(-0.579442159473,-0.37348930198)--(-0.569403793448,-0.37348930198);
\draw(-0.553737872803,-0.364921206423)--(-0.579442159473,-0.37348930198);
\draw(-0.553737872803,-0.364921206423)--(-0.536911669391,-0.364921206423);
\draw(-0.506195025591,-0.341149782817)--(-0.553737872803,-0.364921206423);
\draw(-0.472438036344,-0.341149782817)--(-0.487633892851,-0.348747711071);
\draw(-0.45571263817,-0.335574650093)--(-0.4460724716,-0.331718583465);
\draw(-0.472438036344,-0.341149782817)--(-0.45571263817,-0.335574650093);
\draw(-0.506195025591,-0.341149782817)--(-0.472438036344,-0.341149782817);
\draw(-0.38490017946,-0.219854936686)--(-0.506195025591,-0.341149782817);
\draw(-0.362281973967,-0.288604452394)--(-0.37342182646,-0.29528836389);
\draw(-0.341537793435,-0.274774998706)--(-0.362281973967,-0.288604452394);
\draw(-0.341537793435,-0.274774998706)--(-0.324000835751,-0.266006519864);
\draw(-0.286617731415,-0.219854936686)--(-0.341537793435,-0.274774998706);
\draw(-0.255704177846,-0.219929456308)--(-0.267121947655,-0.228492783665);
\draw(-0.255704177846,-0.219929456308)--(-0.245490696318,-0.213120468622);
\draw(-0.224641585033,-0.188866863495)--(-0.255704177846,-0.219929456308);
\draw(-0.184072585769,-0.161820863985)--(-0.20399897292,-0.181747251137);
\draw(-0.184072585769,-0.161820863985)--(-0.155733224936,-0.140566343361);
\draw(-0.224641585033,-0.188866863495)--(-0.184072585769,-0.161820863985);
\draw(-0.286617731415,-0.219854936686)--(-0.224641585033,-0.188866863495);
\draw(-0.38490017946,-0.219854936686)--(-0.286617731415,-0.219854936686);
\draw(-0.38490017946,0.)--(-0.38490017946,-0.219854936686);
\draw(-0.184072585769,0.161820863985)--(-0.155733224936,0.140566343361);
\draw(-0.184072585769,0.161820863985)--(-0.20399897292,0.181747251137);
\draw(-0.224641585033,0.188866863495)--(-0.184072585769,0.161820863985);
\draw(-0.255704177846,0.219929456308)--(-0.245490696318,0.213120468622);
\draw(-0.255704177846,0.219929456308)--(-0.267121947655,0.228492783665);
\draw(-0.224641585033,0.188866863495)--(-0.255704177846,0.219929456308);
\draw(-0.286617731415,0.219854936686)--(-0.224641585033,0.188866863495);
\draw(-0.341537793435,0.274774998706)--(-0.324000835751,0.266006519864);
\draw(-0.362281973967,0.288604452394)--(-0.37342182646,0.29528836389);
\draw(-0.341537793435,0.274774998706)--(-0.362281973967,0.288604452394);
\draw(-0.286617731415,0.219854936686)--(-0.341537793435,0.274774998706);
\draw(-0.38490017946,0.219854936686)--(-0.286617731415,0.219854936686);
\draw(-0.45571263817,0.335574650093)--(-0.4460724716,0.331718583465);
\draw(-0.472438036344,0.341149782817)--(-0.45571263817,0.335574650093);
\draw(-0.472438036344,0.341149782817)--(-0.487633892851,0.348747711071);
\draw(-0.506195025591,0.341149782817)--(-0.472438036344,0.341149782817);
\draw(-0.553737872803,0.364921206423)--(-0.536911669391,0.364921206423);
\draw(-0.579442159473,0.37348930198)--(-0.569403793448,0.37348930198);
\draw(-0.579442159473,0.37348930198)--(-0.595588726937,0.377525943846);
\draw(-0.553737872803,0.364921206423)--(-0.579442159473,0.37348930198);
\draw(-0.506195025591,0.341149782817)--(-0.553737872803,0.364921206423);
\draw(-0.38490017946,0.219854936686)--(-0.506195025591,0.341149782817);
\draw(-0.38490017946,0.)--(-0.38490017946,0.219854936686);

\end{scope}
\end{tikzpicture}
\caption{Tropical caustic of the bounded connected component of the complement to the real affine nodal cubic given by the equation $y^2=x^2(x+1)$.}
\label{fig-cubic}
\end{figure}
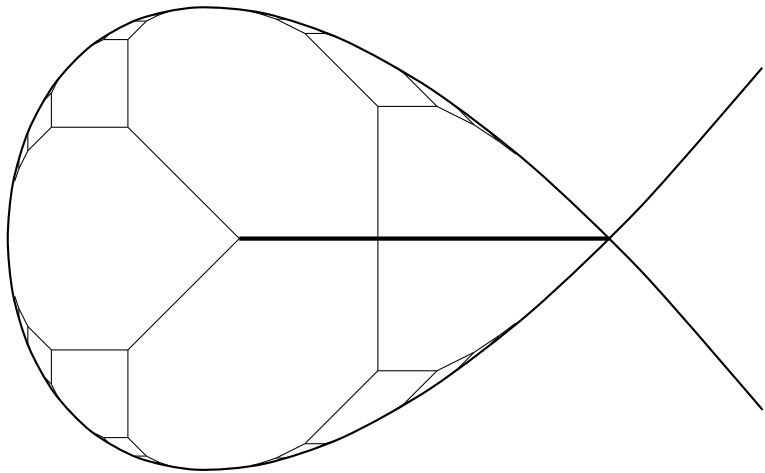
\input Fig-amoeba

Finally, Figure \ref{fig-amoeba} depicts the three unbounded caustics in the three components of the complement of the amoeba of a (trinomial) line in the plane. 

\ \\

\noindent{\bf\Large Part IV. Tropical geometry of continued fractions.}
\section{Continued fractions via tropical angles and their caustics}
Let us recall the notion of {\em tropical cotangent} from the beginning of our paper. If $R_1\subset\nr$ is a ray of rational slope, and $R_2\subset\nr$ is any ray (both emanating from the origin $O\in\nr$) then we can consider the primitive vector $e_1\in N$ parallel to the ray $R_1$ and any primitive vector $e_2\in N$ such that $(e_1,e_2)$ is a basis of $N$ and the intersection $(e_2+ R_1)\cap R_2=\{e_2+se_1\}$ is non-empty.
The choice of $e_2$ is well defined up to adding an integer multiple of $e_1$, thus its residue module 1 us well-defined. We set
$$
ta(R_1R_2)=-s\in\R/\Z.
$$
This quantity completely characterizes the {\em oriented tropical angle} $(R_1R_2)$, i.e. the cone $\Sigma\subset\nr$ generated by $R_1$ and $R_2$ with a choice of a boundary ray $R_1$ is $R_1$ is of rational slope. We may restate the reasoning above as the following statement.
\begin{proposition}
For every $s\in\R/\Z$ there exists a unique oriented tropical angles $(R_1R_2)$ such that $R_1$ is a ray of rational slope and $ta(R_1R_2)=-s$.
\end{proposition}


If $R_2$ is also of rational slope (i.e. $\Sigma$ is a rational tropical angle) then $ta(R_1R_2)\in\mathbb{Q}\slash\mathbb{Z},$ 
and we can exchange the roles of $R_1$ and $R_2$.

\begin{proposition} 
Suppose that $\Sigma\subset\nr$ is rational and $$ta(R_1R_2)={m\over n} \pmod 1$$ (for coprime $m,n\in\Z$, $m\neq 0$, $n>0$).
Then $n$ is the determinant of the tropical angle $\Sigma$, and $$ta(R_2R_1)={k\over n} \pmod 1,$$
where $mk=1 \pmod n.$
\label{prop_reverserays}
\end{proposition} 
Clearly we have $ta(R_1R_2)=0\in\R/\Z$ if and only if $\Sigma$ is a tropical right angle. In this case we also have $ta(R_2R_1)=0$.
\begin{proof}
Let $\nu_1,\nu_2\in N$ be the primitive vectors parallel to the rays $R_1,R_2$.
The equality $ta(R_1R_2)={m\over n} \pmod 1$ is equivalent to divisibility of $\nu_1-m\nu_2$ by $n$.
But then $$k(\nu_1-m\nu_2)=k\nu_1-\nu_2\pmod n,$$
and thus $ta(R_2R_1)={k\over n} \pmod 1$.
\end{proof}

We say that a tropical angle $\Sigma$ generated by rays $R_1$ and $R_2$ is {\em symmetric} if there exists a tropical automorphism exchanging the rays $R_1$ and $R_2$. 
\begin{corollary}
A rational tropical angle is symmetric if and only if $m^2=1\pmod n$, where $n$ is the determinant of $\Sigma$ and $ta(R_1R_2)=\frac mn\pmod 1$.
\end{corollary}
If $m^2=1\pmod n$, we set $ta(\Sigma)=ta(R_1R_2)=ta(R_2R_1)$.


The following statement is an immediate corollary of the definition of $A_{n-1}$-angles and their complementary, canonical, angles.
\begin{proposition} 
A rational tropical angle $\Sigma\subset\nr$ is an $A_{n-1}$-angle if and only if $ta(\Sigma)=\frac {n-1}n\pmod 1$.
It is a canonical  angle if and only if $ta(\Sigma)=\frac 1n\pmod 1$.
\end{proposition}

By Corollary \ref{An-single}, we can describe all tropical angles with the caustics consisting of single rays by numbers $\{\frac 1n\}_{n\in\N}\subset\R/\Z$.
It turns out that we can recover the caustics and their weights from $ta(R_1R_2)$ with the help of continued fractions.

\begin{theorem}Let $\Sigma$ be a cone with boundary rays $R_1$ and $R_2$ such that $R_1$ has a rational slope. Then, $$ta(R_1R_2)=\cfrac{1}{w_1+\cfrac{1}{s_1+\cfrac{1}{w_2+\cfrac{1}{s_2+
   \cdots\vphantom{\cfrac{1}{1}} }}}}\pmod 1,$$
where $w_1,\ w_2,\dots$ is a sequence of weights of rays $W_1,W_2,\dots$ in the tropical caustic $\mathcal{K}_\Sigma$ listed in the natural order starting from the one closest to $R_1$ and 
$s_j$ is the length gradient of the edge $E_j(t)$ in the tropical wave front $\Sigma(t)$ connecting $W_{j}$ and $W_{j+1}$.
\label{prop_regcontfrac}
\end{theorem}

If the second ray $R_2$ also has a rational slope, then this continued fraction is a finite expression with an odd total number of denominators.
Given $\alpha\in (0,1)$ we develop its continued fraction series by subtracting from $\alpha^{-1}$ its floor function $\lfloor \alpha^{-1}\rfloor$ (rounding down), and continuing with $\alpha^{-1}-\lfloor \alpha^{-1}\rfloor$.
If we start with a rational number then, at the last stage, $\alpha^{-1}$ is an integer number $n$. If needed, this number can also be expressed it as $(n-1)+\frac11$ to ensure that the number of denominators in the resulting continued fraction is odd. 

Alternatively, instead of rounding down in such a process, we can use the rounding up, i.e. the ceiling function $\lceil \alpha^{-1}\rceil$ to extract a positive integer from $\alpha$, and continue with $\lceil \alpha^{-1}\rceil-\alpha^{-1}$.
In the resulting continued fraction, all the pluses get replaced with minuses. Such continued fractions are known as Hirzebruch-Jung continued fractions and appear in the study of the resolutions of toric singularities. Indeed, a wave front can be seen as a partial resolution of singularities of a corresponding toric surface, where we keep only $A_n$ singularities. An edge $E(t)$ of the wave front $\Sigma(t)$ of length gradient $s$ corresponds to a boundary divisor with the value of the canonical class equal to $s>0$. The full minimal resolution of singularities would further resolve every $A_n$ singularity into a chain of $n$ divisors with the zero value of the canonical class (i.e. $(-2)$-spheres).
The following statement can be viewed as a restatement of Hirzebruch's theorem \cite{Hirz}.

\begin{theorem}
Let $\Sigma$ be a tropical angle different from the right angle with boundary rays $R_1$ and $R_2$ such that $R_1$ has a rational slope. Then,
$$ta(R_1R_2)=-\cfrac{1}{i_1-\cfrac{1}{i_2-\cfrac{1}{i_3-
   \cdots\vphantom{\cfrac{1}{1}} }}}\mod 1,
 $$
where $-i_1,-i_2,-i_3,\dots\le -2$ is the sequence of self-intersection numbers of the exceptional curves in the minimal resolution of singularities of the toric surface corresponding to the cone $\Sigma$ listed in the order starting from $R_1$. \footnote{Recall that in this paper we use the symplectic toric correspondence (moment map) rather than the spectrum of the semigroup correspondence (fan chars), so that the canonical tropical angle corresponds to the toric $A_n$-singularity, and not its dual $A_n$-angle.}
\label{prop_HJcontfrac}
\end{theorem}

The sequences of denominators of the two types of continued fractions are explicitly related in the following way. The sequence  $i_1,i_2,\dots$ is obtained from the sequence $w_1,s_1,w_2,s_2,\dots$ by replacing every $w_k$ with $w_k-1$ number of entries ``$2$" and every $s_k$ with a single entry ``$s_k+2$".

\begin{example}\label{example_fractions}
Let $\Sigma$ be the cone depicted on the left-hand side of Fig. \ref{1-3-caustic}, and oriented so that $R_1$ is its horizontal ray.
Then we have $$ta(R_1R_2)={4\over 7} \mod 1.$$ 
The caustic $\ks$ consists of two rays, one of weight $1$ and another of weight $3,$ while the wave front $\Sigma(t)$ has a single finite edge with the length gradient $1.$ Theorem \ref{prop_regcontfrac} agrees with the decomposition
$${4\over 7} =\cfrac{1}{1+\cfrac{1}{1+\cfrac{1}{3}}}\ \ \ ,$$ while
Theorem \ref{prop_HJcontfrac} agrees with the decomposition
$$-{3\over 7} =-\cfrac{1}{3-\cfrac{1}{2-\cfrac{1}{2}}}\ \ \ .$$ 
We get $w_1=1$, $s_1=1$, $w_2=3$, accordingly, $w_1=1$ does not contribute to the sequence $i_k$,
$i_1=s_1+2=3$, while $w_2=3$ contributes to $i_2=i_3=2$, see Fig. \ref{fig_examplefractions}.
Reversing the order of denominators in these fractions corresponds to reversing the orientation of the angle $\Sigma$.
We get $ta(R_2R_1)={2\over 7}\mod 1,$ which agrees with Proposition \ref{prop_reverserays} since $2\cdot 4=1\mod 7.$

Figure \ref{2fractions} illustrates the first continued fractions in this example by the picture on the left, and the second continued fraction by the one on the right.
Note the link between the two pictures. The broken line on the right is derived from the dotted broken line on the left, i.e. from the wave front of the corresponding tropical angle by resolution of the $A_{w-1}$-singularities (in this case, the $A_2$-singularity from the caustic ray of weight 3, this yields two $(-2)$-divisors), $i_2=i_3=2$, The $A_0$-singularity from the caustic ray of weight 1 is a smooth point, and does not contribute anything in the Hirzebruch-Jung continued fraction. Finally, the length gradient of the only finite edge of the dotted broken curve is $s_1=1$, this means that the value of the canonical class on the corresponding divisor is 1. Accordingly, its self-intersection is $-2-s_1=-3$, so that $i_1=3$.

The dashed lines on the right picture depict the caustic of the broken line. They are parallel to the caustic rays on the left. The weight 3 caustic ray yields three parallel (vertical) dashed lines on the right. These three lines bound two edges of the broken line of zero length gradient (thus of self-intersection $-2$). The weight 1 caustic ray corresponds to a single dashed line on the right. Together with the rightmost vertical dashed line it bounds the edge of the length gradient equal to 1 (and thus of self-intersection $-3$).
\end{example}
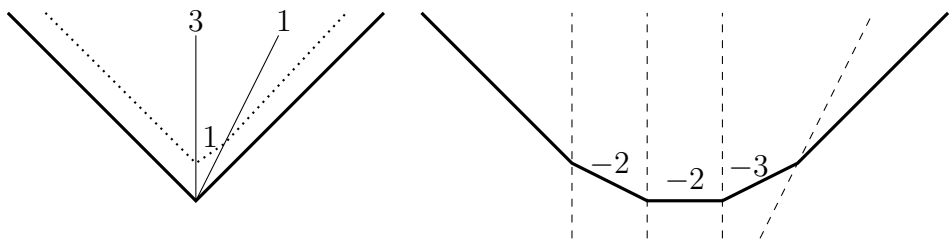
\begin{figure}
\label{fig_examplefractions}
\begin{tikzpicture}
\draw[very thick](-2.5,2.5)--(0,0)--(2.5,2.5);
\draw(0,0)--(0,2.2);
\draw(0,0)--(1.1,2.2);
\draw[thick, dotted](-2,2.5)--(0,0.5)--(0.38,0.76)--(2,2.5);
\node at (0,2.4) {$3$};
\node at (0.19,0.85) {$1$};
\node at (1.17,2.4) {$1$};
\draw[very thick](3,2.5)--(5,0.5)--(6,0)--(7,0)--(8,0.5)--(10,2.5);
\draw[dashed](5,-0.5)--(5,2.5);
\draw[dashed](6,-0.5)--(6,2.5);
\draw[dashed](7,-0.5)--(7,2.5);
\draw[dashed](7.5,-0.5)--(9,2.5);
\node at (5.5,0.49) {$-2$};
\node at (6.5,0.27) {$-2$};
\node at (7.35,0.45) {$-3$};
\end{tikzpicture}
\caption{\label{2fractions} A schematic depiction of denominators for continued fractions of two types in Example \ref{example_fractions}. On the left: the numbers are caustic ray weights and wave front length gradients; the denominators $1,1,3$ of the continued fraction of the first type are combined by alternation of these two sequences. On the right: the numbers denote self-intersections of the exceptional divisors in the complete resolution of the singularity; the denominators $3,2,2$ of the continued fraction of the second type are negatives of this sequence. In both cases, the numbers should be read counterclockwise. }
\end{figure}





\begin{proof}[Proof of Theorem \ref{prop_regcontfrac}]
If $\Sigma$ is a canonical angle, the theorem holds, since $ta(\Sigma)=\frac1n$, and $\ks$ consists of a single ray of weight $n$.

Otherwise, choose coordinates on $N_\mathbb{R}$, such that the boundary ray $R_1$ is the negative part of the ordinate and that the second ray $R_2$ belongs to the positive quadrant and has a slope $\alpha\in(0,1).$ Expand $\alpha$ in a regular continued fraction with a sequence of denominators $w_1,s_1,w_2,s_2,\dots$ which is infinite if $\alpha$ is irrational, or has an odd length if $\alpha$ is rational. 

Even tail truncations of this continued fraction correspond to the slopes of vertices of the tropical wave front of $\Sigma$ at time $1.$ We start this sequence of slopes with $0,$ which corresponds to the leftmost vertex $(1,0)$ of $\Sigma(1)$. 
The next slope is $$\cfrac{1}{w_1+\cfrac{1}{s_1}}$$ corresponding to the vertex $(w_1s_1+1,s_1).$  The edge connecting these two vertices is $(w_1s_1,s_1)$ and has the tropical length $s_1$ is a bounded edge of $\Sigma(1)$ of length gradient $s_1$.


The odd truncations of the continued fraction of $\alpha$ correspond to the slopes of vertices of a polygonal domain $P$ defined as the convex hull of all non-zero lattice points in the cone spanned by the negative ray of the abscissa axis and $R_2.$ 
The leftmost vertex $(0,1)$ of $P$ is followed by $(w_1,1)$
of slope ${1\over w_1}$.
The edge connecting these vertices is
parallel to the first ray of the caustic of $\Sigma$, while its tropical length is the weight of this ray. 


Denote by $p_0,p_1,\dots$ the sequence of all vertices of $P$ where their slopes decrease. Clearly, $p_0=(0,1),$ the sequence is infinite if $\alpha$ is irrational and finite if $\alpha$ is rational the slope of the last one is $\alpha,$ otherwise the sequence is finite with the last vertex having the slope $\alpha.$ Similarly, denote by $q_0,q_1,\dots$ the sequence of all vertices of $\Sigma(1)$ where their slopes increase. The first vertex $q_0$ is $(1,0)$ and the sequence is infinite if and only if $\alpha$ is irrational. Denote by $U_n$ the matrix with the first column equal to the transpose of $q_n$ and the second column equal to the transpose of $p_n.$  In particular, $U_0$ is the identity matrix.
\begin{lemma}\label{lem_regfractionmatrix} There is the following recursion
$$U_n=U_{n-1}\begin{pmatrix}w_n & 1 \\ 1 & 0  \end{pmatrix}\begin{pmatrix}s_n & 1 \\ 1 & 0  \end{pmatrix}.$$
In particular, $U_n$ is a unimodular matrix.
\end{lemma}
This is a version of the key recursive relation in Arnold's seminal exposition \cite{Arnold-cf} of continued fraction for high-school students (see also references therein). It is true by direct computation for $n=1$, while its verification for $n+1$ gets reduced to the same computation after changing the basis to the vectors given by $q_n$ and $p_n$ (which form a basis by induction). 

Thus the tropical lengths of a side of $P$ adjacent to vertices $p_{n-1}$ and $p_{n}$ is the same as the tropical length of the difference between the second columns in the identity matrix and in the product $$\begin{pmatrix}w_n & 1 \\ 1 & 0  \end{pmatrix}\begin{pmatrix}s_n & 1 \\ 1 & 0  \end{pmatrix}=\begin{pmatrix}w_ns_n+1 & w_n \\ s_n & 1  \end{pmatrix},$$ which gives $w_n.$ Similarly, the side of $\Sigma(1)$ adajcent to $q_{n-1}$ and $q_n$ has tropical length $s_n,$ computed from the difference between the first columns of the identity matrix and the matrix above. \end{proof}

\begin{remark}
The proof of Theorem \ref{prop_regcontfrac} exhibits the duality between the weights of the vertices of the wave front $\Sigma(1)$, and the length gradients of the bounded edges of $\Sigma^*(1)$, where $\Sigma^*$ is the dual (or complementary) angle, and vice versa.
This duality breaks for the first and the last number in the sequence (due to the odd length of the resulting continuous fraction expansion). It can be remedied with the help of the identity
$$1-{1\over {w_1+\beta}}=\cfrac{1}{1+\cfrac{1}{w_1-1+\beta}}.$$
Namely, if $w_1$ is greater than one, then we put $1$ for the first weight of the dual cone.
Then the first length gradient of the wave front of the dual cone is $w_1-1.$
If $w_1=1$ then one performs a reversed operation, seeing the first weight of the dual cone to be $s_1+1.$
The only two exceptions are the tropical angle of determinant 1 and 2 (the right angle and the $A_1$-angle) that are self-dual.
If $\Sigma$ is rational then we should do a similar procedure for the last weight in the sequence. See Figure \ref{fig_exdualangle} for an example.
\end{remark}
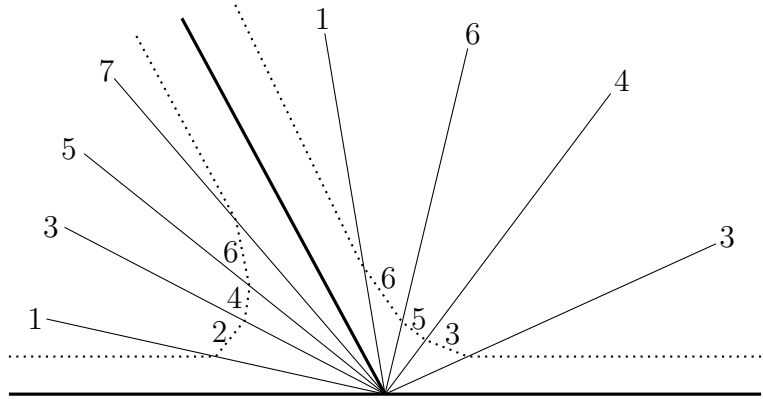
\begin{figure}
\begin{tikzpicture}
\draw[very thick](-5,0)--(5,0);
\draw[very thick](0,0)--(-2.7,5);

\draw(0,0)--(-4.5,1);
\draw(0,0)--(-4.26,2.22);
\draw(0,0)--(-4,3.2);
\draw(0,0)--(-3.6,4.2);
\draw[dotted, thick](-5,0.5)--(-2.25,0.5)--(-1.866,0.982)--(-1.8,1.42)--(-1.98,2.31)--(-3.33,4.81);

\node at (-4.65,1) {1};
\node at (-4.44,2.22) {3};
\node at (-4.2,3.26) {5};
\node at (-3.7,4.3) {7};

\node at (-2.2,0.83) {2};
\node at (-2,1.28) {4};
\node at (-2.05,1.9) {6};

\draw(0,0)--(-0.8,4.8);
\draw(0,0)--(1.1,4.6);
\draw(0,0)--(3,4);
\draw(0,0)--(4.4,2);

\draw[dotted,thick](5,0.5)--(1.1,0.5)--(0.54,0.72)--(0.229,0.958)--(-0.325,1.75)--(-2,5.2);

\node at (-0.83,5) {1};
\node at (1.17,4.8) {6};
\node at (3.15,4.17) {4};
\node at (4.55,2.1) {3};

\node at (0.9,0.8) {3};
\node at (0.45,1) {5};
\node at (0.05,1.55) {6};

\end{tikzpicture}
\caption{A schematic example of the duality between a corner and its complimentary at the level of weights of caustic rays and wave front gradients. The wave fronts at time one are depicted with dotted broken lines, the numbers beside finite segments are their tropical lengths.}
\label{fig_exdualangle}
\end{figure}

\begin{example}
If $d>2$ then we have $w_1=1,\ s_1=d-2,\ w_2=1$ for the $A_{d-1}$-angle, and simply $w_1=d$ for its dual, canonical angle.
\end{example}

\begin{proof}[Proof of Theorem \ref{prop_HJcontfrac}]
Choose coordinates in such a way that the ray $R_1$ is the negative part of the ordinate and $R_2$ belongs to the right half-plane having slope $\alpha\in(-1,0).$ Let $i_1,i_2,\dots$ be the denominators of the Hirzebruch-Jung continued fraction expansion of $\alpha.$ Consider the complementary cone spanned by rays $-R_1$ and $R_2,$ and let $P$ be the convex hull of non-zero lattice points in this cone. Consider the sequence $z_{-1},z_0,z_1,\dots$ of boundary lattice points of $P$ adjacent to finite edges in the order of decreasing slopes. It starts with $z_{-1}=(0,1)$ and $z_0=(1,0).$ 

As in the proof of Theorem \ref{prop_regcontfrac}, we show inductively that $z_{k-1}$, $z_k$ form a basis of $N$, and that $z_{k-2}+z_k=i_kz_{k-1}$.
Indeed, it is straightforward to verify that $z_1=(-1,i_1)$ is a lattice point on the boundary of $P$, and thus $z_0$, $z_1$ form a basis of $N$. After the coordinate change 
$$I_k=I_{k-1}\begin{pmatrix}i_k & 1 \\ -1 & 0  \end{pmatrix},$$
where $I_0$ is the identity matrix we reduce finding the next lattice point $z_{k+1}$ at $\dd P$ to finding the point $z_1$ at the first step. In the new coordinates, the ceiling function of the inverse to minus the slope of $z_{k+1}$ gives $i_{k+1}>1$.

The rays generated by $z_0$, $z_1$, \dots, give a subdivision of the tropical angle $(-R_1R_2)$ complimentary to $\Sigma$ (and thus isomorphic to $\Sigma^*$) into tropical right angles.
In the fan chart toric geometry correspondence, it yields a resolution of the toric surface given by $\Sigma^*$ via fans, i.e. by $\Sigma$ in the moment map language.
This resolution is minimal, as the self-intersection of the $k$th toric divisor in the resolution is $-i_k\le-2$.

\end{proof}

\bibliography{b}
\bibliographystyle{plain}

\end{document}